\documentclass[a4paper,12pt]{amsart}
\usepackage[english]{babel}
\usepackage[utf8]{inputenc} 
\usepackage{amssymb,amsthm,amsmath} 
\usepackage{indentfirst}
\usepackage[makeroom]{cancel}
\usepackage{microtype}
\usepackage{comment}
\usepackage[dvipsnames]{xcolor}
\usepackage{enumerate} 
\usepackage[all]{xy}
\usepackage{tikz} 
\usetikzlibrary{cd} 
\usepackage{bbm} 
\usepackage{mathdots}
\usepackage[hidelinks,draft=false]{hyperref}
\usepackage[colorinlistoftodos,textsize=small,textwidth=80]{todonotes}
\usepackage{braket}
\usepackage[toc]{appendix}
\usepackage{mathtools}
\usepackage{cite}
\usepackage{stmaryrd} 
\usepackage[left=4cm, right=4cm, top=2cm]{geometry}
\usepackage{epsfig}
\usepackage[matha,mathb]{mathabx}
\def\acts{\circlearrowright } 

\hypersetup{
    linktoc=all,     
}

\usepackage[normalem]{ulem}

\makeatletter
\DeclareFontEncoding{LS1}{}{}
\DeclareFontSubstitution{LS1}{stix}{m}{n}
\DeclareMathAlphabet{\mathscr}{LS1}{stixscr}{m}{n}
\makeatother

\newtheorem{theorem}{Theorem}[section]
\newtheorem{lemma}[theorem]{Lemma}
\newtheorem*{lemma*}{Lemma}

\newtheorem{proposition}[theorem]{Proposition}

\newtheorem{corollary}[theorem]{Corollary}

\theoremstyle{definition}
\newtheorem{definition}[theorem]{Definition}
\newtheorem{lemma-definition}[theorem]{Lemma/Definition}
\theoremstyle{definition}

\theoremstyle{remark}
\newtheorem{remark}[theorem]{Remark}
\newtheorem{example}[theorem]{Example}

\numberwithin{equation}{section}

\newcommand{\C}{\mathbb{C}}
\newcommand{\R}{\mathbb{R}}
\newcommand{\Z}{\mathbb{Z}}

\definecolor{lightgrey}{rgb}{0.7, 0.7, 0.7}
\definecolor{ballblue}{rgb}{0.0, 0.5, 1.0}
\definecolor{purple}{rgb}{0.0, 0.0, 0.0}

\renewcommand{\phi}{\varphi}

\renewcommand{\Im}{\mathfrak{Im}}
\renewcommand{\Re}{\mathfrak{Re}}

\DeclareMathOperator{\rk}{rk}

\title[$\C/\Lambda$-equivariant cohomology and the Witten class]{The (anti-)holomorphic sector in $\C/\Lambda$-equivariant cohomology, and the Witten class}
\author{Mattia Coloma}
\address{Università degli Studi di Roma "Tor Vergata"; Dipartimento di Matematica, Via della Ricerca Scientifica, 1 - 00133 - Roma, Italy; 
}
\email{coloma@mat.uniroma2.it}
\author{Domenico Fiorenza}
\address{Sapienza Universit\`a di Roma; Dipartimento di Matematica ``Guido Castelnuovo'', P.le Aldo Moro, 5 - 00185 - Roma, Italy; 
}
\email{fiorenza@mat.uniroma1.it}
\author{Eugenio Landi}
\address{Università di Roma Tre; Dipartimento di Matematica e Fisica, Largo San Leonardo Murialdo, 1 - 00146 - Roma, Italy;
}
\email{eugenio.landi@uniroma3.it}
\date{\today}

\begin{document}
\maketitle

\begin{abstract}
Atiyah's classical work on circular symmetry and stationary phase shows how the $\hat{A}$-genus is obtained by formally applying the equivariant cohomology  localization formula to the loop space of a simply connected spin manifold. The same technique, applied to a suitable ``antiholomorphic sector'' in the $\C/\Lambda$-equivariant cohomology of the conformal double loop space $\mathrm{Maps}(\C/\Lambda,X)$ of a rationally string manifold $X$ produces the Witten genus of $X$. This can be seen as an equivariant localization counterpart to Berwick-Evans supersymmetric localization derivation of the Witten genus.
\end{abstract}

\begin{quote}
\begin{flushright}
    \emph{Se vogliamo che tutto rimanga come è, bisogna che tutto cambi.}
\end{flushright}
\end{quote}

\tableofcontents

\section{Introduction}
In the classic work \cite{atiyahcircsym}, Atiyah shows how to recover the $\hat{A}$-class of a compact smooth spin\footnote{{\color{purple}The spin structure on $X$ is needed in order to have an orientation on the free loop space $\mathrm{Maps}(\mathbb{T},X)$. For the constructions in the present paper, since we are not concerned with integration over loop spaces, the weaker assumption that $X$ is orientable will suffice. 
}}
manifold $X$ via a formal infinite dimensional version of the Duistermaat-Heckman formula applied to the smooth loop space $\mathrm{Maps}(\mathbb{T},X)$ of maps from a circle to $X$. Such a formula is a particular case of the well known localization formula for torus equivariant cohomology, extensively treated in \cite{ATIYAHmoment}. The appearance of the $\hat{A}$-class in such an infinite dimensional version of localization techniques in torus equivariant cohomology was pointed out by Atiyah as ``a brilliant observation of the physicist E. Witten'' and suggests that, reasoning as in \cite{atiyahcircsym}, the Witten class $\mathrm{Wit}(X)$ \cite{cmp/1104117076,Witten:1987cg}, should 
emerge from a localization formula for the torus equivariant cohomology of the double loop space $\mathrm{Maps}(\mathbb{T}^2,X)$ of maps from a 2-dimensional torus to $X$. This is indeed the case, as long as one makes an a priori unjustified assumption: that the generators $u,v$ of the $\mathbb{T}^2$-equivariant cohomology of a point over $\C$, 
\[
H_{\mathbb{T}^2}^*(\mathrm{pt};\C) \cong \C[u,v],
\]
are not independent but rather satisfy a $\C$-linear dependence condition of the form
\[
v=\tau u
\]
where $\tau$ is a point in the complex upper half plane $\mathbb{H}$, see \cite{lu2008regularized}. Although the hypothesis of $\C$-linear dependence of the polynomial variables $u,v$ may appear somewhat ``ad hoc'' to make the computations work out, it suggests that if instead of looking at a topological torus $\mathbb{T}^2$ we consider a complex torus $\C/\Lambda$ then there should exist a version of the localization theorem for torus equivariant cohomology, where only a holomorphic variable $\xi$ (or its conjugate $\overline{\xi}$) appears, instead of the two real variables $u,v$. In this paper we show that such a {\it holomorphic} ({resp. \it antiholomorphic}) {\it sector} of the $\C/\Lambda$-equivariant cohomology can indeed be defined and that an (anti-)holomorphic localization formula holds. Going back to what inspired it, in the final part of the paper we show how the Witten class of a compact smooth manifold emerges from the antiholomorphic localization formula for the $\C/\Lambda$-equivariant cohomology of the double loop space $\mathrm{Maps}(\C/\Lambda,X)$. 
More precisely, the idea is to formally apply the finite-dimensional antiholomorphic localization formula obtained in the first part of the paper to the inclusion of $X$ in $\mathrm{Maps}(\C/\Lambda,X)$ as the submanifold of constant maps (that are the fixed points for the $\C/\Lambda$-translation action on $\mathrm{Maps}(\C/\Lambda,X)$). It turns out, however, that the infinite products that would naively define the equivariant Euler class for the normal bundle $\nu_\Lambda$  of $X$ in $\mathrm{Maps}(\C/\Lambda,X)$ do not converge, so a suitable $\zeta$-regularization is needed in order to make sense of these infinite products. Once this is done, one obtains that if $X$ is a compact rational string manifold, i.e., if $X$ is a compact manifold with torsion first Pontryagin class, then the inverse of the normalized Euler class of $\nu_\Lambda$ defines a modular form with values in the complex cohomology of $X$, which turns out to be the Witten class of $X$. In particular, the integral over $X$ of the inverse of the normalized Euler class of $\nu_\Lambda$ is the Witten genus of $X$. From the point of view of $\C/\Lambda$-equivariant cohomology, the geometric condition that $X$ needs to be a rationally string manifold will emerge as the condition ensuring that the $\zeta$-regularization procedure involved in the infinite rank localization formula is independent of the choice of arguments for the nonzero elements in the lattice $\Lambda\subset \C$. This condition will also imply that the expected modular properties of the inverse normalized Euler class are not disrupted by the $\zeta$-regularization.

Equivariant cohomology and the Atiyah--Bott localization formula admit an elegant rephrasing in terms of supergeometry.\footnote{{\color{purple}This is a classical idea whose origin can be traced back to the foundational works in supersymmetric field theories from the 70s and the 80s. It has then become popular in the mathematical community through purely mathematical applications such as Witten's \cite{witten1982supersymmetry}, and has been a main source of inspiration for Atiyah and Bott's \cite{ATIYAHmoment}, which is in turn one of the main sources for this note.  See, e.g., \cite{blau-thompson}, \cite{schwarz-zaboronsky}, and \cite{pestun-introduction-to-localization-in-quantum-field-theory} for  modern surveys from the point of view of supersymmetric field theories.}} Reversing this point of view, every differential geometric construction obtained through supersymmetric localization techniques in quantum field theory should in principle admit a derivation internal to the setting of equivariant cohomology. In this sense, the results of this paper can be seen as an equivariant localization counterpart to Berwick-Evans supersymmetric localization derivation of the Witten genus \cite{berwickevans2013witten, berwickevans2019supersymmetric}, with the Weierstra{\ss} $\zeta$-regularization of equivariant Euler classes playing the role of the $\zeta$-regularization of infinite dimensional determinants in supersymmetric quantum field theory.
\vskip .5 cm

We thank the Referee for very useful comments and suggestions that helped us in improving both the content and the exposition of the paper. D.F.'s research has been partially supported by PRIN 2017 --- 2017YRA3LK.

\section{1d Euclidean tori equivariant cohomology}
\subsection{The Euclidean Cartan complex for circle actions}
As a half-way step towards two dimensional real tori endowed with a complex structure $\C/\Lambda$, we start by recalling a few basic constructions in the equivariant cohomology for 1-dimensional torus actions, formulating them for 1-dimensional Euclidean tori $\R/\Lambda$ rather than for the topological 1-dimensional torus $\mathbb{T}$. Here $\Lambda \subset \R$ is a lattice in $\R$, i.e. an additive subgroup of $\R$ isomorphic to $\Z$.

 The quotient $\R/\Lambda$ can be thought of as a circle of length $\ell$, with $\ell$ the minimum strictly positive element of $\Lambda$. It has a natural structure of real Lie group; we will denote its Lie algebra by $\mathfrak{t}_{\Lambda}$.
Next we consider a compact smooth manifold $M$ with a smooth action of $\R/\Lambda$,
 denote by $\Omega^\bullet(M;\R)^{\R/\Lambda}$ the $\R/\Lambda$-invariant part of the de Rham algebra of $M$, and endow 
 \[
\Omega^\bullet(M;\R)^{\R/\Lambda}\otimes_\R\mathrm{Sym}({\mathfrak{t}_\Lambda}^\vee[-2])
\]
with a bigrading where the component of bidegree $(k,l)$ is $\Omega^{k-l}(M;\R)^{\R/\Lambda}\otimes_\R\mathrm{Sym}^l({\mathfrak{t}_\Lambda}^\vee[-2])$. 
This bigraded vector space comes equipped with a structure of bicomplex where the differential of degree $(1,0)$ is the de Rham differential (acting trivially on $\mathrm{Sym}({\mathfrak{t}_\Lambda}^\vee[-2])$) and the differential of degree $(0,1)$ is the operator ${e_\Lambda}^\vee[-2] \iota_{v_{e_\Lambda}}$, where $({e_\Lambda},{e_\Lambda}^\vee)$ is a pair consisting of a linear generator of $\mathfrak{t}_\Lambda$ and of its dual element in ${\mathfrak{t}}_\Lambda^\vee$, and $v_{e_\Lambda}$ is the vector field on $M$ corresponding to $e_\Lambda$ via the differential of the action. The operator $\iota$ is the contraction operator. It is immediate to see that ${e_\Lambda}^\vee[-2] \iota_{v_{e_\Lambda}}$ is independent of the choice of the generator $e_\Lambda$. 
\begin{definition}
The Cartan complex of $\R/\Lambda\acts M$ is the total complex of the bicomplex 
\[
(\Omega^\bullet(M;\R)^{\R/\Lambda}\otimes_\R\mathrm{Sym}({\mathfrak{t}_\Lambda}^\vee[-2]);d_{\mathrm{dR}},{e_\Lambda}^\vee[-2] \iota_{v_{e_\Lambda}}).
\]
The total differential in the Cartan complex is denoted by $d_{\R/\Lambda}$ and is called the equivariant differential.
Elements in the Cartan complex that are $d_{\R/\Lambda}$-closed are called \emph{equivariantly closed} forms.
\end{definition}
\begin{remark}
The importance of the Cartan complex resides in the fact its cohomology is the real $\R/\Lambda$-equivariant cohomology $H^\bullet_{\R/\Lambda}(M;\R)$ of $M$. So it provides a differential geometric tool to compute this cohomology. It is the generalization to the equivariant setting of the de Rham complex computing real singular cohomology.
\end{remark}
\begin{remark}
Evaluation at $0\in \mathfrak{t}^{\vee}_\Lambda[-2]$ is a morphism of complexes from the Cartan complex to the de Rham complex $(\Omega^\bullet(M;\R)^{\R/\Lambda},d_{\mathrm{dR}})$ of $\R/\Lambda$-invariant forms. One says that an element $\widetilde{\omega}$ in the Cartan complex is an extension of an invariant form $\omega$ if $\widetilde{\omega}\vert_0=\omega$.
\end{remark}
\begin{remark}
The quotient map $\R\to \R/\Lambda$ gives a distinguished Lie algebra isomorphism $\mathrm{Lie}(\R)\xrightarrow{\sim} \mathfrak{t}_\Lambda$. By means of this isomorphism, the Cartan bicomplex is isomorphic to 
\[
(\Omega^\bullet(M;\R)^{\R/\Lambda}[u];d_{\mathrm{dR}}, u\iota_{v_{d/dx}})
\]
where $d/dx$ is the standard basis vector in $\mathrm{Lie}(\R)$ and $u$ is a degree 2 formal variable corresponding to the dual 1-form $dx$ placed in degree 2. Notice that with respect to the bigrading, the variable $u$ has bidegree $(1,1)$.
\end{remark}

{\color{purple}\begin{remark}\label{rem:grading-convention}
With the usual bigrading on the Cartan complex, i.e., with $\Omega^{k}(M;\R)^{\R/\Lambda}\otimes_\R\mathrm{Sym}^l({\mathfrak{t}_\Lambda}^\vee[-2])$ in bidegree $(k,2l)$, the differentials $d_{\mathrm{dR}}$ and ${e_\Lambda}^\vee[-2] \iota_{v_{e_\Lambda}}$ have bidegree $(1,0)$ and $(-1,2)$, respectively. We made a linear change in the grading so that they have bidegree $(1,0)$ and $(0,1)$, respectively.
\end{remark}}

\subsection{$\R/\Lambda$-equivariant characteristic classes}\label{R-equivariant-characteristic-classes}
Equivariant vector bundles over an $\R/\Lambda$-manifold come with a natural notion of equivariant characteristic classes. When the action on the manifold is trivial\footnote{This does not imply that the action is trivial on the bundle.}, equivariant characteristic classes admit a simple combinatorial/representation theoretic description that we recall below. 
\begin{remark}
A typical situation where one meets equivariant vector bundles on a $\R/\Lambda$-trivial base is by considering equivariant vector bundles on the $\R/\Lambda$-fixed point locus $\mathrm{Fix}(M)$ in an $\R/\Lambda$-manifold $M$. Notice that, since $\R/\Lambda$ is a compact Lie group, its action on $M$ is automatically proper and so $\mathrm{Fix}(M)$ is a smooth submanifold of $M$.  Equivariant vector bundles on $\mathrm{Fix}(M)$ need not be restrictions of equivariant vector bundles on $M$. A classical example is the normal bundle $\nu$ for the inclusion $\mathrm{Fix}(M)\hookrightarrow M$.
\end{remark}
For ease of exposition, we will tacitly assume $\mathrm{Fix}(M)$ to be connected: in the more general situation of a possibly nonconnected fixed point locus all the constructions we recall in this section are to be repeated for each of the connected components of $\mathrm{Fix}(M)$. 
\begin{remark}
For a $\R/\Lambda$-trivial manifold one has $M=\mathrm{Fix}(M)$, so it is actually not restrictive to work with submanifolds of the form $\mathrm{Fix}(M)$ when one is interested into equivariant vector bundles over $\R/\Lambda$-trivial base manifolds.
\end{remark}
\begin{remark}
As the $\R/\Lambda$-action is trivial on $\mathrm{Fix}(M)$, the associated Cartan bicomplex is
\[
(\Omega^\bullet(\mathrm{Fix}(M);\R)\otimes_\R \mathrm{Sym}({\mathfrak{t}_\Lambda}^\vee[-2]);d_{\mathrm{dR}}, 0)
\]
and so the $\R/\Lambda$-equivariant cohomology of $\mathrm{Fix}(M)$ is
\[
H^\bullet_{\R/\Lambda}(\mathrm{Fix}(M);\R)=H^\bullet(\mathrm{Fix}(M);\R)\otimes_\R \mathrm{Sym}({\mathfrak{t}_\Lambda}^\vee[-2]).
\]
\end{remark}
The key to the combinatorial description of equivariant complex vector bundles over $\R/\Lambda$-trivial base manifolds is the following statement, which is an immediate consequence of the regularity of the decomposition into isotypic components of smooth families of complex representations of compact Lie groups. The statement is well known, see, e.g., \cite[Proposition 4.6]{sinha} where however it is given without a proof. For completeness, we provide a proof for the particular case we are interested in.
\begin{lemma}
An $\R/\Lambda$-equivariant complex line bundle on $\mathrm{Fix}(M)$ is equivalently the datum of a pair $(L,\chi)$, where $L$ is a complex line bundle on $\mathrm{Fix}(M)$ and $\chi\colon \R/\Lambda\to U(1)$ is a character of $\R/\Lambda$. 
\end{lemma}
\begin{proof}
Let us denote by $L_p$ the fiber of $L$ on the point $p\in M$.
The datum of an $\R/\Lambda$-equivariant complex line bundle $L$ on $\mathrm{Fix}(M)$ is the datum of a collection of group homomorphisms 
$\R/\Lambda\to \mathrm{Aut}_{\mathbb{C}}(L_p)$, smoothly depending on $p\in \mathrm{Fix}(M)$. Since $L_p$ is 1-dimensional, one has a canonical isomorphism $\mathrm{Aut}_{\mathbb{C}}(L_p)\cong \mathbb{C}^\ast$, so our datum is the datum of a smooth family of Lie group homomorphisms $\chi_p\colon \R/\Lambda\to \mathbb{C}^\ast$. Since $\R/\Lambda$ is compact, these have to factor through $U(1)$ and so they form a smooth family of characters $\chi_p\colon \R/\Lambda\to U(1)$. Since $U(1)$-valued characters of $\R/\Lambda$ are uniquely defined by their topological degree as smooth maps $\R/\Lambda\to U(1)$ and the topological degree is a homotopy invariant, we have that $\chi_p$ is constant on connected components of $\mathrm{Fix}(M)$. Therefore, if $\mathrm{Fix}(M)$ is connected, as we are assuming, we are reduced with the datum of a single $U(1)$-valued character $\chi$ of $\R/\Lambda$.
\end{proof}

By the above Lemma, in what follows we will write an $\R/\Lambda$-equivariant complex line bundle over $\mathrm{Fix}(M)$ as a pair $(L,\chi)$.
\begin{definition}
Let $\chi\colon \R/\Lambda\to U(1)$ be a character. The \emph{weight} of  $\chi$ is the linear map 
\[
w_\chi\colon \mathfrak{t}_\Lambda\to \mathbb{R}
\]
defined as follows:
$2\pi i w_\chi$ is the Lie algebra homomorphism
$2\pi i w_\chi\colon \mathfrak{t}_\Lambda \to 2\pi i \R=\mathrm{Lie}(U(1))$
associated with the Lie group homomorphism $\chi$, i.e., $2\pi i w_\chi$ is the linear map making the diagram
\[
\begin{tikzcd}
\mathfrak{t}_\Lambda \cong \R \arrow[d, "\mathrm{proj}"'] \arrow[r, "2\pi i w_{\chi}"] & 2\pi i \R \arrow[d, "\mathrm{exp}(-)"] \\
\R /\Lambda \cong \R/\lambda\Z \arrow[r, "\chi"']                                                 & U(1)                                  
\end{tikzcd}
\]
commute. 
\end{definition}
\begin{remark}
Notice that, by definition, $w_{\chi}$ is an element of ${\mathfrak{t}_\Lambda}^\vee$, and so $w_{\chi}[2]\in {\mathfrak{t}_\Lambda}^\vee[-2]\subseteq \mathrm{Sym}({\mathfrak{t}_\Lambda}^\vee[-2])$.
\end{remark}

\begin{definition}
Let $(L,\chi)$ be an $\R/\Lambda$-equivariant complex line bundle over $\mathrm{Fix}(M)$.
The equivariant first Chern class of $(L,\chi)$ is the element of $H^\bullet(\mathrm{Fix}(M);\R)\otimes_\R \mathrm{Sym}({\mathfrak{t}_\Lambda}^\vee[-2])$ given by
\[
c_{1,\R/\Lambda}(L,\chi):=c_1(L)+w_\chi[-2].
\]
\end{definition}
\begin{remark}
 It is convenient to give a more explicit description of $c_{1,\R/\Lambda}(L,\chi)$ in terms of the isomorphism 
 \[
 H^\bullet_{\R/\Lambda}(\mathrm{Fix}(M);\R)\cong H^\bullet(\mathrm{Fix}(M);\R)[u]
 \]
 induced by the Lie algebra isomorphism $\mathrm{Lie}(\R)\xrightarrow{\sim}\mathfrak{t}_\Lambda$. In order to do so, recall that characters of $\R/\Lambda$ are indexed by the dual lattice $\Lambda^\vee$ of $\Lambda$ and that
via the standard inner product in $\R$ this is identified with $\Lambda$: 
every character of $\R/\Lambda$ is of the form 
\[
\chi(x) = \rho_\lambda(x) := \mathrm{exp}(2\pi i \lambda \mathrm{vol}(\R/\Lambda)^{-2}x),
\]
for some $\lambda \in \Lambda$. The associated weight $w_\lambda$ is then $w_\lambda=\lambda\mathrm{vol}(\R/\Lambda)^{-2}  dx$ so that $w_\lambda [-2]=\lambda \mathrm{vol}(\R/\Lambda)^{-2}u$.
The equivariant first Chern class of $(L,\rho_\lambda)$ is then written as $
c_{1,\R/\Lambda}(L,\rho_\lambda)=c_1(L)+\lambda \mathrm{vol}(\R/\Lambda)^{-2}u$.
Introducing the rescaled formal variable $u_\Lambda:=\mathrm{vol}(\R/\Lambda)^{-2}u$, of the same bidegree as $u$, this is written
\[
c_{1,\R/\Lambda}(L,\rho_\lambda)=c_1(L)+\lambda u_\Lambda.
\]
\end{remark}
For a $\R/\Lambda$-equivariant complex vector bundle $E$ on $\mathrm{Fix}(M)$ one defines the equivariant Chern classes of $E$ by the equivariant splitting principle. Namely, first one decomposes $E$ as the direct sum of its isotypic components,
\[
E=\bigoplus_{\chi\in \Lambda^\vee}E_\chi;
\]
next, one defines the equivariant Chern roots of each component $E_\chi$ via the splitting principle:
\[
\{\alpha_{i,\R/\Lambda}(E_\chi)\}_{i=1,\dots \mathrm{rk}E_\chi}
=
\{\alpha_{i}(E_\chi)+w_\chi[-2]\}_{i=1,\dots \mathrm{rk}E_\chi},
\]
where the $\alpha_{i}(E_\chi)$'s are the Chern roots of $E_\chi$. Finally one defines the total $\R/\Lambda$-equivariant Chern class of $E$ by means of these equivariant Chern roots.
\begin{definition}
In the same notation as above, the total $\R/\Lambda$-equivariant Chern class of $E$
is 
\[
c_{\R/\Lambda}(E):=\prod_{\chi\in \Lambda^\vee}c_{\R/\Lambda}(E_\chi),
\]
with
\[
c_{\R/\Lambda}(E_\chi):=\prod_{i=1}^{\mathrm{rk} E_\chi}(1+\alpha_{i,\R/\Lambda}(E_\chi)).
\]
In particular, the top $\R/\Lambda$-equivariant Chern class of $E$ is
\[
c_{\mathrm{top},\R/\Lambda}(E)=\prod_{\chi\in \Lambda^\vee}\prod_{i=1}^{\mathrm{rk} E_\chi}(\alpha_{i}(E_\chi)+w_\chi[-2]).
\]
\end{definition}
{\color{purple}
\begin{remark}
All through the paper we will denote by $\mathrm{rk}E$ the complex rank of a complex vector bundle and by $\rk_{\R}(E)$ the real rank of a real vector bundle (and of the real vector bundle underlying a complex vector bundle).
 Notice that for $\mathrm{rk}E_{\chi}=0$ the product $\prod_{i=1}^{\mathrm{rk} E_{\chi}}$ is an empty product and so it is 1 by definition. That is, only the $\chi$'s with $\mathrm{rk}E_{\chi}\geq 1$ actually contribute to the above products.
 \end{remark}}

\begin{remark}
In terms of the formal variable $u_\Lambda$ and the identification between $\Lambda^\vee$ and $\Lambda$, these read

\[
c_{\R/\Lambda}(E)=\prod_{\lambda\in \Lambda}\prod_{i=1}^{\mathrm{rk} E_{\rho_\lambda}}(1+\alpha_{i,\R/\Lambda}(E_{\rho_\lambda})+\lambda u_\Lambda)
\]
and
\[
c_{\mathrm{top},\R/\Lambda}(E)=\prod_{\lambda\in \Lambda}\prod_{i=1}^{\mathrm{rk} E_{\rho_\lambda}}(\alpha_{i}(E_{\rho_\lambda })+\lambda u_\Lambda).
\]
\end{remark}
\bigskip

It is convenient to isolate the contribution from the isotypic component of the trivial character $\mathbf{0}\in \Lambda^\vee$, corresponding to the zero weight. We write
\[
E=E_{\mathbf{0}}\oplus E^{\mathrm{eff}}=E_{\mathbf{0}}\oplus\bigoplus_{\chi\in \Lambda^\vee\setminus \{\mathbf{0}\}}E_\chi=
E_{\mathbf{0}}\oplus\bigoplus_{\lambda\in \Lambda\setminus\{0\}}E_{\rho_\lambda},
\]
and call $E^{\mathrm{eff}}$ the \emph{effectively acted} bundle. By multiplicativity of the total Chern class and of the top Chern class one finds
\[
c_{\R/\Lambda}(E)=c(E_{\mathbf{0}})c_{\R/\Lambda}(E^{\mathrm{eff}}); \qquad
c_{\mathrm{top},\R/\Lambda}(E)=c_{\mathrm{top}}(E_{\mathbf{0}})c_{\mathrm{top},\R/\Lambda}(E^{\mathrm{eff}}).
\]
\begin{definition}
The \emph{weight polynomial} of $E^{\mathrm{eff}}$ is the element in $\mathrm{Sym}({\mathfrak{t}_\Lambda}^\vee[-2])$ given by
\[
wp(E^{\mathrm{eff}}):=\prod_{\chi\in \Lambda^\vee\setminus\{\mathbf{0}\}}(w_\chi[-2])^{\mathrm{rk}E_\chi}=
\prod_{\chi\in \Lambda^\vee\setminus\{\mathbf{0}\}}w_\chi^{\mathrm{rk}E_\chi}[-2\mathrm{rk} E].
\]
\end{definition}
\begin{remark}
By construction, the weight polynomial $wp(E^{\mathrm{eff}})$ is a nonzero element in $\mathrm{Sym}({\mathfrak{t}_\Lambda}^\vee[-2])$.
\end{remark}
By localizing the $\R/\Lambda$-equivariant cohomology of $\mathrm{Fix}(M)$ at $wp(E^{\mathrm{eff}})$, i.e., by formally inverting $wp(E^{\mathrm{eff}})$ one can rewrite the top $\R/\Lambda$-equivariant Chern class of $E^{\mathrm{eff}}$ as
\begin{equation}\label{eq:top-chern}
c_{\mathrm{top},\R/\Lambda}(E^{\mathrm{eff}})= wp(E^\mathrm{eff})\widehat{c_{\mathrm{top},\R/\Lambda}}(E^{\mathrm{eff}}),
\end{equation}
where
\begin{equation}\label{eq:normalized-top-chern}
\widehat{c_{\mathrm{top},\R/\Lambda}}(E^{\mathrm{eff}})
:=\prod_{\chi\in \Lambda^\vee\setminus\{\mathbf{0\}}}\prod_{i=1}^{\mathrm{rk} E_\chi}\left(1+\frac{\alpha_{i}(E_\chi)}{w_\chi[-2]}\right).
\end{equation}
\begin{definition}
The degree zero element $\widehat{c_{\mathrm{top},\R/\Lambda}}(E^{\mathrm{eff}})$ in the localization  $H^\bullet_{\R/\Lambda}(\mathrm{Fix}(M);\R)_{(wp(E^{\mathrm{eff}}))}$ is called the \emph{normalized top Chern class} of $E^{\mathrm{eff}}$.
\end{definition}
\begin{remark}
Notice that $\widehat{c_{\mathrm{top},\R/\Lambda}}(E^{\mathrm{eff}})$ is an invertible element in the localization  $H^\bullet_{\R/\Lambda}(\mathrm{Fix}(M);\R)_{(wp(E^{\mathrm{eff}}))}$.
\end{remark}
\begin{remark}

Equivalently, in terms of the variable $u_\Lambda$ one writes
\[c_{\mathrm{top},\R/\Lambda}(E^\mathrm{eff})=u_\Lambda^{\mathrm{rk} E^\mathrm{eff}}
\left(\prod_{\lambda\in \Lambda\setminus\{0\}}\lambda ^{\mathrm{rk} E_{\rho_\lambda}} \right) 
\underbrace{\prod_{\lambda\in \Lambda\setminus\{0\}}\prod_{i=1}^{\mathrm{rk} E_{\rho_\lambda}}\left(1+\frac{\alpha_{i}(E_{\rho_\lambda})u_\Lambda^{-1}}{ \lambda}\right)}_{\widehat{c_{\mathrm{top},\R/\Lambda}}(E^{\mathrm{eff}})}.
\]
\end{remark}

\subsection{Equivariant Euler classes of real vector bundles}\label{R-equivariant-euler-class}
Since equivariant vector bundles come naturally with a notion of equivariant characteristic classes, real oriented equivariant vector bundles come with a natural notion of equivariant Euler class. And again, if the equivariant vector bundle has a {\color{purple}base space with trivial $\R/\Lambda$-action}, the combinatorics behind the computation of an equivariant Euler class is purely representation theoretic.
\par 
Real irreducible representations of $\R/\Lambda$ are indexed by the quotient set $\Lambda^\vee/\pm$. The unique fixed point $\mathbf{0}$ corresponds to the trivial representation, which is the unique 1-dimensional real representation of $\R/\Lambda$; the equivalence class $[\chi]$ of the complex character $\chi$ corresponds to the irreducible real 2-dimensional representation $\chi_\R$. As $\chi^{-1}\cong \overline{\chi}$, we see that $(\chi^{-1})_\R$ and $\chi_\R$ are isomorphic as real representations. In terms of the distinguished isomorphism of $\Lambda^\vee\cong \Lambda$ induced by the inner product, the involution on $\Lambda^\vee$ reads $\lambda\leftrightarrow -\lambda$ and the above isomorphism of complex characters is $\rho_{-\lambda}\cong \overline{\rho_\lambda}$. In particular, we see that every  nontrivial irreducible real representation of $\R/\Lambda$ factors through 
a complex character via the standard inclusion $U(1)\cong SO(2)\hookrightarrow O(2)$:
\[
\begin{tikzcd}
\R/\Lambda \arrow[r, "\chi"'] \arrow[rr, "\phi", bend left] & U(1) \arrow[r] & O(2)
\end{tikzcd}
\]
As a consequence, if we decompose an $\R/\Lambda$-equivariant real vector bundle $V$ over $\mathrm{Fix}(M)$ as
\[
V = V_{[\mathbf{0}]}\oplus{V^{\mathrm{eff}}} = V_{[\mathbf{0}]}\oplus \bigoplus_{[\chi] \in \Lambda^\vee\setminus{\{\mathbf{0}\}}/\pm}V_{[\chi]},
\]
we see that the effective component $V^{\mathrm{eff}}$ can always be endowed (non canonically) with a complex structure. In particular $V^{\mathrm{eff}}$ is always an even rank orientable vector bundle. 
\begin{remark}
By choosing an orientation for $V^{\mathrm{eff}}$ one has a well defined equivariant Euler class for it, and a change in the choice of the orientation corresponds to a sign change in the equivariant Euler class. 
\end{remark}
The above remark leads to the following doubling trick. The two possible equivariant Euler classes for $V^{\mathrm{eff}}$, corresponding to the two possible orientations, are precisely the two solutions of the equation
\begin{equation}\label{eq:equation-for-eul}
[\omega]^2=(-1)^\frac{\rk V^\mathrm{eff}}{2}c_{\mathrm{top},\R/\Lambda}(V^\mathrm{eff} \otimes \C)
\end{equation}
with $[\omega]$ of degree $\frac{1}{2}\rk_\R V^{\mathrm{eff}}$. The choice of one solution then determines an orientation of $V^{\mathrm{eff}}${\color{purple}; namely, the orientation whose corresponding equivariant Euler class is the chosen solution}. 
\begin{remark}\label{rem:pfaffian}
Since characteristic classes with real or complex coefficients can be computed via Chern-Weil theory, equation (\ref{eq:equation-for-eul}) has a simple origin in linear algebra: if $F_\nabla\in \Omega^2(M,\mathfrak{so}(2k))$ is the curvature 2-form for a Riemannian connection $\nabla$ on an even rank orientable vector bundle $V$ on a smooth manifold $M$, then the top Chern class of $V\otimes \mathbb{C}$  has a closed form representative given by the determinant $\det(\frac{1}{2\pi i}F_\nabla)=(-1)^{k}\det(\frac{1}{2\pi}F_\nabla)$, while the Euler class of $V$ has a closed form representative given by the Pfaffian $\mathrm{Pf}(\frac{1}{2\pi}F_\nabla)$, and for any skew-symmetric matrix $A$ in $\mathfrak{so}(2k)$ one has $\mathrm{Pf}(A)^2=\det(A)$.
\end{remark}
\begin{definition}\label{def:equivariant-euler1}
Let a choice of arguments for the elements $\lambda\in \Lambda\setminus\{0\}$ be fixed. The equivariant Euler class $\mathrm{eul}_{\R/\Lambda}(V^\mathrm{eff})$ defined by this choice is the 
distinguished solution of equation (\ref{eq:equation-for-eul}), given by
\begin{equation}\label{eul:topchern}
\mathrm{eul}_{\R/\Lambda}(V^\mathrm{eff}) := (iu_\Lambda)^{\frac{\rk V^{\mathrm{eff}}}{2}}\left(\prod_{\lambda\in \Lambda\setminus\{0\}}\lambda ^{\frac{\mathrm{rk} (V^{\mathrm{eff}}\otimes\C)_{\rho_\lambda}}{2}} \right)\sqrt{\widehat{c_{\mathrm{top},\R/\Lambda}}(V^\mathrm{eff} \otimes \C)},
\end{equation}
where the determination of the square root is such that $\sqrt{1+t}=1+t/2+\cdots$. The distinguished orientation on $V^\mathrm{eff}$ defined by the given choice of arguments is the one that is coherent with this choice of equivariant Euler class.
\end{definition}
\begin{remark}
Since we are assuming $V$ is a finite rank vector bundle, only finitely many ranks $\rk(V^{\mathrm{eff}}\otimes \C)_{\rho_\lambda}$ are nonzero. So the product in (\ref{eul:topchern}) is actually a finite product and one actually only needs to choose arguments for the finitely many $\lambda$'s in $\Lambda\setminus 0$ such that $\rk(V^{\mathrm{eff}}\otimes \C)_{\rho_\lambda}$ is nonzero.
\end{remark}

\begin{definition}\label{def:normalized-equivariant-euler1}
The $(\R/\Lambda)$-equivariant cohomology class 
\[
\widehat{\mathrm{eul}_{\R/\Lambda}}(V^\mathrm{eff}):=\sqrt{\widehat{c_{\mathrm{top},\R/\Lambda}}(V^\mathrm{eff} \otimes \C)}
\]
is called the normalized equivariant Euler class of $V^\mathrm{eff}$.
\end{definition}
\begin{remark}
The normalized equivariant Euler class $\widehat{\mathrm{eul}_{\R/\Lambda}}(V^\mathrm{eff})$ 
is independent of any choice of arguments, and so is canonically associated with the real equivariant vector bundle $V$.
\end{remark}

\begin{remark}
If the $\R/\Lambda$-equivariant vector bundle $V$ is oriented, one endows $V_{[\mathbf{0}]}$ with the orientation compatible with those of $V$ and $V^{\mathrm{eff}}$. By this procedure, applied to the tangent bundle of an oriented $\R/\Lambda$-manifold $M$, one gets a canonical orientation for $\mathrm{Fix}(M)$ once a choice of arguments for the nonzero elements in the lattice $\Lambda$ has been fixed.
\end{remark}

\section{The (anti-)holomorphic sector for a complex torus action}
With this short reminder of equivariant cohomology for 1d Euclidean tori actions, we have set the stage to describe the Cartan complex and equivariant cohomology classes for the action of 2d flat tori equipped with a complex structure.
\par
By definition, these tori are given by the quotients $\C/\Lambda$ of $\C$ by two dimensional lattices $\Lambda \subset \C$, so they are the natural generalization of the Euclidean 1d tori $\R/\Lambda$ considered in the previous section. The quotients $\C/\Lambda$ have a natural structure of Lie groups and, as in the 1d case, we will denote their Lie algebra by $\mathfrak{t}_\Lambda$. Moreover $\C/\Lambda$, carries a holomorphic structure compatible with the group addition, so that complex tori are an example of holomorphic Lie groups. This gives the Lie algebra $\mathfrak{t}_\Lambda$ a complex Lie algebra structure that will allow us to give the complexified Cartan complex of a $\C/\Lambda$-action a holomorphic kick. The following statement is immediate.
\begin{lemma}\label{lemma:decomposition}
Let $M$ be a compact smooth manifold $M$ equipped with a smooth action by $\C/\Lambda\acts{M}$. The complex structure on $\mathfrak{t}_\Lambda$ gives a natural splitting $\mathfrak{t}_\Lambda \otimes_\R \C \cong {\mathfrak{t}_\Lambda}^{1,0} \oplus {\mathfrak{t}_\Lambda}^{0,1}$ inducing a decomposition
\[
\Omega^\bullet(M;\C)^{\C/\Lambda} \otimes_\C \mathrm{Sym}(({\mathfrak{t}_\Lambda}^{1,0})^\vee[-2]) \otimes_\C \mathrm{Sym}(({\mathfrak{t}_\Lambda}^{0,1})^\vee[-2]),
\]
of the complexified Cartan complex
\[
\Omega^\bullet(M;\C)^{\C/\Lambda} \otimes_\C \mathrm{Sym}((\mathfrak{t}_\Lambda \otimes_\R \C)^\vee[-2])
\]
 computing the equivariant cohomology of $\C/\Lambda \acts{M}$ with coefficients in $\C$.
This  realizes the complexified Cartan complex as the total complex of a tricomplex with
 $\Omega^{k-p-q}(M;\C)^{\C/\Lambda}\otimes_\C\mathrm{Sym}^p({\mathfrak{t}^{1,0}_\Lambda}^\vee[-2])\otimes_\C \mathrm{Sym}^q({\mathfrak{t}^{0,1}_\Lambda}^\vee[-2])$ in tridegree $(k,p,q)$. The differential of degree $(1,0,0)$ in this tricomplex is the de Rham differential; the differential of degree $(0,1,0)$ is the operator ${e_\Lambda}^\vee[-2] \iota_{v_{e_\Lambda}}$, where $({e_\Lambda},{e_\Lambda}^\vee)$ is a pair consisting of a $\C$-linear generator of ${\mathfrak{t}_\Lambda}^{1,0}$ and of its dual element in $({\mathfrak{t}_\Lambda}^{1,0})^\vee$, and $v_{e_\Lambda}$ is the complex vector field on $M$ corresponding to $e_\Lambda$ via the differential of the action; the differential of degree $(0,0,1)$ is the operator ${\overline{e}_\Lambda}^\vee[-2] \iota_{v_{\overline{e}_\Lambda}}$, where $({\overline{e}_\Lambda},{\overline{e}_\Lambda}^\vee)$ is a pair consisting of a $\C$-linear generator of ${\mathfrak{t}_\Lambda}^{0,1}$ and of its dual element in $({\mathfrak{t}_\Lambda}^{0,1})^\vee$.
\end{lemma} 
\begin{remark}\label{rem:xi-xibar}
The isomorphism $\mathrm{Lie}(\C) \xrightarrow{\sim} \mathfrak{t}_\Lambda$ induced by the projection $\C\to \C/\Lambda$ induces natural $\C$-linear generators for $\mathfrak{t}_\Lambda^{1,0}$ and $\mathfrak{t}_\Lambda^{0,1}$, given by the images of the complex invariant vector fields $\partial/\partial z$ and $\partial/\partial\overline{z}$, respectively. Denoting by $\xi$ and $\overline{\xi}$ the dual invariant 1-forms $dz$ and $d\overline{z}$ placed in degree 2, the complexified Cartan tricomplex is written
\[
(\Omega^\bullet(M;\C)^{\C/\Lambda}[\xi,\overline{\xi}];d_\mathrm{dR}, \xi\iota_{v_{\partial/\partial z}}, \overline{\xi}\iota_{v_{\partial/\partial \overline{z}}}).
\]
With respect to the given  trigrading, the variables $\xi$ and $\overline{\xi}$ have tridegree $(1,1,0)$ and $(1,0,1)$, respectively.
\end{remark}

{\color{purple}\begin{remark}
The somehow unusual trigrading is chosen so that the differentials have tridegree $(1,0,0), (0,1,0)$ and $(0,0,1)$, respectively. See Remark \ref{rem:grading-convention}.
\end{remark}}
By restricting the Cartan tricomplex only to the antiholomorphic (resp. holomorphic) part, i.e. by taking only ${\mathfrak{t}_\Lambda}^{0,1}$ (resp. ${\mathfrak{t}_\Lambda}^{1,0}$) instead of $\mathfrak{t}_\Lambda \otimes_\R \C$, and taking the associated total complex, we end up with the definition of the \emph{antiholomorphic} (resp. \emph{holomorphic}) \emph{sector} of the Cartan complex over $\C$. 
\begin{definition}
In the same assumptions as in Lemma \ref{lemma:decomposition}, the antiholomorphic sector of the complexified Cartan complex is the total complex associated with the bicomplex
\[
(\Omega^\bullet(M;\C)^{\C/\Lambda} \otimes_\C \mathrm{Sym}(({\mathfrak{t}_\Lambda}^{0,1})^\vee[-2]);d_\mathrm{dR},{\overline{e}_\Lambda}^\vee[-2] \iota_{v_{\overline{e}_\Lambda}}).
\]
Its total differential will be denoted by $\overline{\partial}_{\C/\Lambda}$ and its cohomology by the symbol $H^\bullet_{\C/\Lambda;\overline{\partial}}(M;\C)$. By changing ${\mathfrak{t}_\Lambda}^{0,1}$ into ${\mathfrak{t}_\Lambda}^{1,0}$ one obtains the definition of the holomorphic sector.
\end{definition}
\begin{remark}
In terms of the distinguished basis $\{\partial/\partial z,\partial/\partial\overline{z}\}$ of $\mathrm{Lie}(\C)\otimes \C$, the antiholomorphic sector of the Cartan complex over $\C$ is the total complex associated to the bicomplex
\[
(\Omega^\bullet(M;\C)^{\C/\Lambda}[\overline{\xi}];d_\mathrm{dR},\overline{\xi}\iota_{v_{\partial/\partial \overline{z}}}).
\]
\end{remark}
\subsection{$\C/\Lambda$-equivariant Chern classes}
 Exactly as in the $\R/\Lambda$ case, $\C/\Lambda$-equivariant complex line bundles over $\mathrm{Fix}(M)$ are equivalently pairs $(L,\chi)$ consisting of a complex line bundle $L$ over $\mathrm{Fix}(M)$ and a character $\chi: \C/\Lambda \to U(1)$, and the first equivariant Chern class of $(L,\chi)$ in the $\C/\Lambda$-equivariant Cartan complex is 
\[
c_{1,\C/\Lambda}(L,\chi) = c_1(L)+w_\chi[-2],
\]
where $w_\chi$ is the weight of $\chi$, i.e., the $\R$-linear map defined by the commutative diagram
\[
\begin{tikzcd}
\mathfrak{t}_\Lambda \cong \C \arrow[d, "\mathrm{proj}"'] \arrow[r, "2\pi i w_{\chi}"] & 2\pi i \R \arrow[d, "\mathrm{exp}(-)"] \\
\C/\Lambda\arrow[r, "\chi"']                                                          & U(1)                                  
\end{tikzcd}.
\]
Chern classes of higher rank $\C/\Lambda$-equivariant complex vector bundles are defined exactly as in the $\R/\Lambda$ setting: one first decomposes the bundle as the direct sum of its isotypic components, and then formally splits each of these as a direct sum of line bundles. This way one defines the equivariant Euler classes $\mathrm{eul}_{\C/\Lambda}(V^\mathrm{eff})$ and the normalized equivariant Euler classes $\widehat{\mathrm{eul}_{\C/\Lambda}}(V^\mathrm{eff})$  by generalizing Definitions \ref{def:equivariant-euler1} and \ref{def:normalized-equivariant-euler1}.

\begin{remark}\label{rem:notation-rho}
By means of the standard Hermitian pairing on $\C$, the dual lattice $\Lambda^\vee$ of characters of $\C/\Lambda$ is identified with $\Lambda$: every character of $\C/\Lambda$ is of the form
\[
\rho_{\lambda}(z)=\mathrm{exp}\left(\pi \dfrac{\lambda\overline{z}-\overline{\lambda}z}{\mathrm{vol}(\C/\Lambda)}\right),
\]
for some $\lambda\in \Lambda$.
The corresponding weight is 
\[
w_{\lambda}= \dfrac{\lambda d\overline{z}-\overline{\lambda}dz}{2i\mathrm{vol}(\C/\Lambda)}.
\]
The first equivariant Chern class of $(L,\rho_{\lambda})$ is given by
\[
c_{1,\C/\Lambda}(L,\rho_{\lambda}) = c_1(L) + \lambda\overline{\xi}_\Lambda-\overline{\lambda}\xi_\Lambda,
\]
where
\[
 \xi_\Lambda=\frac{\xi}{2i \mathrm{vol}(\C/\Lambda)}; \qquad 
\overline{\xi}_\Lambda=\frac{\overline{\xi}}{2i \mathrm{vol}(\C/\Lambda)}.
\]
\end{remark}
We will be particularly interested in the antiholomorphic part of the $\C/\Lambda$-equivariant Chern classes, i.e. the classes in the antiholomorphic sector obtained by evaluating the holomorphic parameter $\xi$ at $0$. By the splitting principle, these are determined by the antiholomorphic parts of the equivariant first Chern classes,

\begin{equation}
c^{\overline{\partial}}_{1,\C/\Lambda}(L,\rho_{\lambda})=c_{1,\C/\Lambda}(L,\chi)\vert_{\xi=0} = c_1(L) + \lambda\overline{\xi}_\Lambda.
\end{equation}
From this, one has the following immediate generalization of (\ref{eq:top-chern}, \ref{eq:normalized-top-chern}):
\begin{equation}\label{eq:c-top-with-zeta-antiholo}
c^{\overline{\partial}}_{\mathrm{top},\C/\Lambda}(E^\mathrm{eff}) 
=\underbrace{{\overline{\xi}_\Lambda}^{\rk E^\mathrm{eff}}\left(\prod_{\lambda\in \Lambda\setminus\{0\}}\lambda ^{\mathrm{rk} E_{\rho_\lambda}} \right)}_{wp^{\overline{\partial}}(E^{\mathrm{eff}})} \widehat{c_{\mathrm{top},\C/\Lambda}^{\overline{\partial}}}(E^\mathrm{eff}),
\end{equation}
where 
\[
\widehat{c_{\mathrm{top},\C/\Lambda}^{\overline{\partial}}}(E^\mathrm{eff}):=
\prod_{\lambda\in \Lambda\setminus\{0\}}\prod_{i=1}^{\mathrm{rk} E_{\rho_{\lambda}}}\left(1+\frac{\alpha_{i}(E_{\rho_{\lambda}})\overline{\xi}_\Lambda^{-1}}{\lambda}\right).
\]
\begin{remark}
The polynomial 
\[
wp^{\overline{\partial}}(E^{\mathrm{eff}})={\overline{\xi}_\Lambda}^{\rk E^\mathrm{eff}}\left(\prod_{\lambda\in \Lambda\setminus\{0\}}\lambda ^{\mathrm{rk} E_{\rho_\lambda}} \right)
\]
in the variable $\overline{\xi}_\Lambda$ is the weight polynomial of $E^{\mathrm{eff}}$ (or, more precisely, its complexification) evaluated at $\xi=0$. One calls it the antiholomorphic weight polynomial. By construction, it is a nonzero element in $H^\bullet_{\C/\Lambda;\overline{\partial}}(\mathrm{Fix}(M);\C)$.
\end{remark}
\begin{definition}
The degree zero element $\widehat{c_{\mathrm{top},\C/\Lambda}^{\overline{\partial}}}(E^\mathrm{eff})$ in the localization  $H^\bullet_{\C/\Lambda;\overline{\partial}}(\mathrm{Fix}(M);\C)_{(wp^{\overline{\partial}}(E^{\mathrm{eff}}))}$ is called the \emph{normalized antiholomorphic top Chern class} of $E^{\mathrm{eff}}$.
\end{definition}
\begin{remark}\label{rem:why-antiholo}
There is no particular reason to prefer the antiholomorphic sector over the holomorphic sector if not this: when $\Lambda=\Lambda_\tau$ is the lattice $\Z\oplus \Z\tau$, the association
\[
\begin{tikzcd}
\tau \arrow[r, maps to] & {c^{\overline{\partial}}_{1,\C/\Lambda_\tau}(L,\rho_{m+n\tau}) = c_1(L) + \left(m+n\tau\right)\overline{\xi}_{\Lambda_\tau}}
\end{tikzcd}
\]
is holomorphic in terms of the modular parameter $\tau$ rather than in terms of the conjugate parameter $\overline{\tau}$.
\end{remark}
 By analogy with the construction in Section \ref{R-equivariant-euler-class}, for real $\C/\Lambda$-equivariant bundles we have a notion of (normalized) equivariant Euler classes in the antiholomorphic sector for their effectively acted parts. 
 \begin{definition}\label{def:eul-antiholo}
 Let $V$ be a real $\C/\Lambda$-equivariant bundle on $\mathrm{Fix}(M)$ and let $V^{\mathrm{eff}}$ be its effectively acted subbundle. For a fixed choice of the arguments for the elements $\lambda\in \Lambda\setminus\{0\}$, the equivariant Euler class of $V^{\mathrm{eff}}$ in the antiholomorphic sector is the element in $H^\bullet_{\C/\lambda;\overline{\partial}}(\mathrm{Fix}(M);\C)$ defined by
 \[
 \mathrm{eul}^{\overline{\partial}}_{\C/\Lambda}(V^\mathrm{eff})
 =(i\overline{\xi}_\Lambda)^{\frac{\rk V^{\mathrm{eff}}}{2}}\left(\prod_{\lambda\in \Lambda\setminus\{0\}}\lambda ^{\frac{\mathrm{rk} (V^\mathrm{eff} \otimes \C))_{\rho_\lambda}}{2}} \right)\underbrace{\sqrt{\widehat{c^{\overline{\partial}}_{\mathrm{top},\R/\Lambda}}(V^\mathrm{eff} \otimes \C)}}_{\widehat{\mathrm{eul}^{\overline{\partial}}_{\C/\Lambda}}(V^\mathrm{eff})}.
 \]
 The invertible degree zero element $\widehat{\mathrm{eul}^{\overline{\partial}}_{\C/\Lambda}}(V^\mathrm{eff})$ in the localization of $H^\bullet_{\C/\lambda;\overline{\partial}}(\mathrm{Fix}(M);\C)$ at ${wp^{\overline{\partial}}(V^{\mathrm{eff}}\otimes_\R\C)}$ is called the normalized equivariant Euler class of $V^{\mathrm{eff}}$ in the antiholomorphic sector.
  \end{definition}
 \begin{remark}
 The normalized Euler class $\widehat{\mathrm{eul}^{\overline{\partial}}_{\C/\Lambda}}(V^\mathrm{eff})$ in the antiholomorphic sector is independent of the choice of arguments for the elements $\lambda$'s.
 \end{remark}

 \begin{remark}\label{rem:even-powers}
 Since the real vector bundle  $V^{\mathrm{eff}}$ carries a complex structure the nonzero Chern roots of its complexification $V^{\mathrm{eff}}\otimes\C$ come in opposite pairs. From
 \[
 \left(1+\frac{\alpha_{i}(E_{\rho_{\lambda}})\overline{\xi}_\Lambda^{-1}}{\lambda}\right)\left(1-\frac{\alpha_{i}(E_{\rho_{\lambda}})\overline{\xi}_\Lambda^{-1}}{\lambda}\right)=1-\frac{\alpha_{i}(E_{\rho_{\lambda}})^2\overline{\xi}_\Lambda^{-2}}{\lambda^2}
 \]
 we see that only even powers of $\overline{\xi}_\Lambda^{-1}$ appear in the expansion of $\widehat{\mathrm{eul}^{\overline{\partial}}_{\C/\Lambda}}(V^\mathrm{eff})$ as a polynomial in the variable $\overline{\xi}_\Lambda^{-1}$. 
 \end{remark}
 
 The following statement is immediate from the definitions. As it will be used several times in what follows, we make it stand out as a Lemma.
 \begin{lemma}\label{lem:anti-holo-eul}
 In the same notation as in Definition \ref{def:eul-antiholo}, the following identities hold:
 \[
 \mathrm{eul}^{\overline{\partial}}_{\C/\Lambda}(V^\mathrm{eff})=\mathrm{eul}_{\C/\Lambda}(V^\mathrm{eff})\bigr\vert_{\xi=0}
 \]
 and
 \[
 \widehat{\mathrm{eul}^{\overline{\partial}}_{\C/\Lambda}}(V^\mathrm{eff})=\widehat{\mathrm{eul}_{\C/\Lambda}}(V^\mathrm{eff})\bigr\vert_{\xi=0}.
 \]
 \end{lemma}

\section{The antiholomorphic localization theorem}
Localization techniques are a very common and powerful tool in equivariant cohomology. We will briefly recall the main theorem, the Atiyah-Bott localization theorem for a $d$-dimensional torus actions \cite{ATIYAHmoment} stated in its Euclidean version, i.e., for flat tori of the form $\R^d/\Lambda$,  and then show how for complex tori $\C/\Lambda$ the result continues to hold even when we restrict our attention to the antiholomorphic sector.
\subsection{The localization formula for a Euclidean torus actions}
Let  $\R^d/\Lambda$ be a $d$-dimensional Euclidean torus, with Lie algebra $\mathfrak{t}_\Lambda$, and let $M$ be a smooth compact connected oriented finite dimensional manifold endowed with a smooth $\R^d/\Lambda$-action. Assume $\mathrm{Fix}(M)$ is a nonempty smooth submanifold of $M$, and denote by $\nu$ the normal bundle to the inclusion $\iota\colon\mathrm{Fix}(M)\hookrightarrow M$. 
The $\R^d/\Lambda$-action on $\nu$ is completely effective, i.e., $\nu_{\{\mathbf{0}\}}=0$ and so, by the same argument used above in the case $d=1$, the real bundle $\nu$ carries a complex structure. In particular, it is of even rank and orientable.
Once an orientation is fixed, one has a well defined equivariant Euler class for $\nu$, that can be written as
\[
\mathrm{eul}_{\R^d/\Lambda}(\nu)=wp(\nu)\widehat{ \mathrm{eul}_{\R^d/\Lambda}}(\nu)
\]
with $wp(\nu)$ a degree $2\rk \nu$ element in $\mathrm{Sym}(\mathfrak{t}_\Lambda^\vee[-2])$, called the weight polynomial, and $\widehat{ \mathrm{eul}_{\R^d/\Lambda}}(\nu)$ a degree zero invertible element in the $\R^d/\Lambda$-equivariant cohomology of $\mathrm{Fix}(M)$ localized at $wp(\nu)$, of the form $1+\cdots$.
One orients $\mathrm{Fix}(M)$ in such a way that its orientation is compatible with those on $M$ and on $\nu$. 
Having fixed this notation, the Atiyah-Bott localization theorem reads as follows.
\begin{theorem}[Localization isomorphism]\label{thm:localization-iso}
After localization at the weight polynomial $wp(\nu)$, the equivariant cohomologies of $M$ and $\mathrm{Fix}(M)$ become isomorphic  $\mathrm{Sym}(\mathfrak{t}_\Lambda^\vee[-2])_{(wp(\nu))}$-modules. An explicit isomorphism is given by:
\[
\begin{tikzcd}
H_{\R^d/\Lambda}^\bullet(M,\R)_{(wp(\nu))} \arrow[rr, "\mathrm{eul}_{\R^d/\Lambda}(\nu)^{-1}\cdot \iota^*"] &  & {H_{\R^d/\Lambda}^{\bullet}(\mathrm{Fix}(M),\R)_{(wp(\nu))}}[-\mathrm{rk}\nu].
\end{tikzcd}
\]
The inverse isomorphism is given by the equivariant pushforward $\iota_\ast$. \end{theorem}
{\color{purple} 
\begin{remark}
In what follows we are not making use of pushforwards in equivariant cohomology, so the last line in the statement of Theorem \ref{thm:localization-iso} is included only for completeness. We address the interested reader to \cite{ATIYAHmoment} for details on pushforwards in equivariant cohomology.
\end{remark}
}
\begin{remark}
The localization isomorphism is induced by a morphism between the Cartan complexes. To realize such a morphism one only needs to choose closed forms representatives in $\Omega^\bullet(\mathrm{Fix}(M);\C)^{\R^d/\Lambda}$ for the Chern classes of the normal bundle $\nu$, endowed with a chosen complex structure. Such a choice determines a representative for $\mathrm{eul}_{\R^d/\Lambda}(\nu)^{-1}$ in $\Omega^\bullet(\mathrm{Fix}(M);\R)^{\R^d/\Lambda}\otimes_\R \mathrm{Sym}(\mathfrak{t}_\Lambda^\vee[-2])_{(wp(\nu))}$, which we will denote by the same symbol  $\mathrm{eul}_{\R/\Lambda}(\nu)^{-1}$, and one has a morphism of differential graded $\mathrm{Sym}(\mathfrak{t}_\Lambda^\vee[-2])_{(wp(\nu))}$-modules
\[
\begin{tikzcd}
\Omega^\bullet(M,\R)^{\R^d/\Lambda}\otimes_\R\mathrm{Sym}(\mathfrak{t}_\Lambda^\vee[-2])_{(wp(\nu))} \arrow[d, "\mathrm{eul}_{\R^d/\Lambda}(\nu)^{-1}\cdot \iota^*"] 
\\
{\Omega^{\bullet}(\mathrm{Fix}(M),\R)^{\R^d/\Lambda}\otimes_\R\mathrm{Sym}(\mathfrak{t}_\Lambda^\vee[-2])_{(wp(\nu))}}[-\mathrm{rk}\nu].
\end{tikzcd}
\]
The Atiyah-Bott theorem then says that this morphism is a quasi-isomorphism.
\end{remark}
The fact that the inverse of $\mathrm{eul}_{\R^d/\Lambda}(\nu)^{-1}\cdot \iota^*$ is the equivariant pushforward $\iota_*$ has the following important consequence.
\begin{corollary}[Localization formula]\label{cor:localization-formula-general}
Let $\tilde{\omega}\in (\Omega^\bullet(M;\R)^{\R^d/\Lambda})_{wp(\nu)}$ be an equivariantly closed form in the localization of the Cartan complex of $M$. Then
\begin{equation}\label{eq:eq-localization-general}
\int_{M}\tilde{\omega} = \int_{\mathrm{Fix}(M)}\frac{\iota^*\tilde{\omega}}{
\mathrm{eul}_{\R^d/\Lambda}(\nu)}.
\end{equation}
\end{corollary}
Corollary \ref{cor:localization-formula-general} is often used in the following version, to compute integrals of invariant forms on $M$.
\begin{corollary}\label{cor:localization-formula}
Let $\omega\in \Omega^{\dim M}(M;\R)^{\R^d/\Lambda}$ be an invariant top degree form on $M$. Assume one has an equivariantly closed extension $\tilde{\omega}\in (\Omega^\bullet(M,\R)^{\R^d/\Lambda}\otimes\mathrm{Sym}(\mathfrak{t}_\Lambda^\vee[-2]))^{\dim M}$ of $\omega$. Then
\begin{equation}\label{eq:eq-localization}
\int_{M}\omega = \int_{\mathrm{Fix}(M)}\frac{\iota^*\tilde{\omega}}{
\mathrm{eul}_{\R^d/\Lambda}(\nu)}.
\end{equation}
\end{corollary}
\begin{remark}\label{rem:independent}
In the particular setting of Corollary \ref{cor:localization-formula}, the localization formula (\ref{eq:eq-localization}) tells us that the term on its right hand side, which is a priori an element in the $\R$-algebra $\mathrm{Sym}(\mathfrak{t}_\Lambda^\vee[-2])_{(wp(\nu))}$, is actually a constant, i.e., an element of $\R$. Also notice that despite the right-hand side in (\ref{eq:eq-localization}) appears on first sight to depend on the choice of an orientation of $\nu$ it actually does not depend on it, as the orientation of $\mathrm{Fix}(M)$ is not fixed a priori but is determined by that of $\nu$ in such a way that they are jointly compatible with the orientation of $M$.
\end{remark}
\begin{remark}
When $d=1$, one can use (\ref{eul:topchern}) to write the localization formula (\ref{eq:eq-localization}) as
\begin{equation}\label{eq:eq-localization-with-zeta}
\int_{M}\omega =(iu_\Lambda)^{-\frac{\rk \nu}{2}}
\left(\prod_{\lambda\in \Lambda\setminus\{0\}}\lambda ^{-\frac{\mathrm{rk} \nu_{\rho_\lambda}}{2}} \right)\int_{\mathrm{Fix}(M)}\frac{\iota^*\widetilde{\omega}}{
{\widehat{\mathrm{eul}_{\mathrm{top},\R/\Lambda}}(\nu)}}.
\end{equation}
The right hand side of (\ref{eq:eq-localization-with-zeta}) a priori depends on the choice of the arguments for the elements $\lambda\in \Lambda\setminus\{0\}$, and it is actually independent of it due to the same argument as in remark \ref{rem:independent}.
\end{remark}

\subsection{The antiholomorphic localization theorem}
Let us now consider complex tori $\C/\Lambda$. In this situation, Theorem \ref{thm:localization-iso} becomes the following.
\begin{theorem}[Localization Isomorphism in the Antiholomorphic Sector]\label{thm:localization-iso-antihol}
After localization at the antiholomorphic weight polynomial $wp^{\overline{\partial}}(\nu)$, the antiholomorphic sectors of equivariant cohomologies of $M$ and $\mathrm{Fix}(M)$ become isomorphic $\C[\overline{\xi}]_{(wp^{\overline{\partial}}(\nu))}$-modules. An explicit isomorphism is given by:
\[
\begin{tikzcd}
H_{\C/\Lambda;\overline{\partial}}^\bullet(M,\C)_{(wp^{\overline{\partial}}(\nu))} \arrow[rr, "\mathrm{eul}_{\C/\Lambda}^{\overline{\partial}}(\nu)^{-1}\cdot \iota^*"] &  & {H_{\C/\Lambda;\overline{\partial}}^{\bullet}(\mathrm{Fix}(M),\C)_{(wp^{\overline{\partial}}(\nu))}}[-\mathrm{rk}\nu].
\end{tikzcd}
\]
The inverse isomorphism is given by the restriction of the equivariant pushforward $\iota_*$ to the antiholomorphic sector.
\end{theorem}
\begin{proof}

In terms of the distinguished variables $\xi$ and $\overline{\xi}$ introduced in Remark \ref{rem:xi-xibar},
 the localization quasi-isomorphism is written as the quasi isomorphism of differential graded $\C[\xi,\overline{\xi}]_{(wp(\nu))}$-modules
\[
\begin{tikzcd}
\Omega^\bullet(M,\C)^{\C/\Lambda}[\xi,\overline{\xi}]_{(wp(\nu))} \arrow[rr, "\mathrm{eul}_{\C/\Lambda}(\nu)^{-1}\cdot \iota^*"] 
&&
{\Omega^{\bullet}(\mathrm{Fix}(M),\C)^{\C/\Lambda}[\xi,\overline{\xi}]_{(wp(\nu))}}[-\mathrm{rk}\nu].
\end{tikzcd}
\]
Evaluation at $\xi=0$ induces a surjective homomorphism 
\[
\C[\xi,\overline{\xi}]_{(wp(\nu))}\xrightarrow{\vert_{\xi=0}} \C[\overline{\xi}]_{(wp^{\overline{\partial}}(\nu))}.
\]
From this we get the morphism of short exact sequences of complexes
\[
\begin{tikzcd}
0 \arrow[d]                                                                                                                                                               &  & 0 \arrow[d]                                                                                             \\
{\xi\Omega^\bullet(M;\C)^{\C/\Lambda}[\xi,\overline{\xi}]_{(wp(\nu))}} \arrow[d] \arrow[rr, "\mathrm{eul}_{\C/\Lambda}(\nu)^{-1}\cdot \iota^*"]                                                                                  &  & {\xi\Omega^\bullet(\mathrm{Fix}(M);\C)[\xi,\overline{\xi}]_{(wp(\nu))}[-\mathrm{rk}\nu]} \arrow[d]                                                                              \\
{\Omega^\bullet(M;\C)^{\C/\Lambda}[\xi,\overline{\xi}]_{(wp(\nu))}} \arrow[rr, "\mathrm{eul}_{\C/\Lambda}(\nu)^{-1}\cdot \iota^*"] \arrow[d, "\vert_{\xi=0}"']     &  & {\Omega^\bullet(\mathrm{Fix}(M);\C)[\xi,\overline{\xi}]_{(wp(\nu))}[-\mathrm{rk}\nu]} \arrow[d, "\vert_{\xi=0}"] \\
{\Omega^\bullet(M;\C)^{\C/\Lambda}[\overline{\xi}]_{(wp^{\overline{\partial}}(\nu))}} \arrow[rr, "\mathrm{eul}^{\overline{\partial}}_{\C/\Lambda}(\nu)^{-1}\cdot \iota^*"] \arrow[d] &  & {\Omega^\bullet(\mathrm{Fix}(M);\C)[\overline{\xi}]_{(wp^{\overline{\partial}}(\nu))}[-\mathrm{rk}\nu]} \arrow[d]         \\
0                                                                                                                                                                         &  & 0,                                                                                                      
\end{tikzcd}
\]
where the commutativity of the bottom square follows from Lemma \ref{lem:anti-holo-eul}.

Since the first two horizontal arrows are quasi-isomorphisms by the Atiyah-Bott localization theorem, so is the third one. This proves the first part of the statement. Since the differential on the Cartan complexes for the fixed point loci reduces to the de Rham differential acting trivially on the variables $\xi,\overline{\xi}$, the induced linear map
\[
H_{\C/\Lambda}^{\bullet}(\mathrm{Fix}(M),\C)_{(wp(\nu))}\xrightarrow{\vert_{\xi=0}}H_{\C/\Lambda;\overline{\partial}}^{\bullet}(\mathrm{Fix}(M),\C)_{(wp^{\overline{\partial}}(\nu))}
\]
is just the evaluation at $\xi=0$. Writing it as 
\[
H^\bullet(\mathrm{Fix}(M),\C)[\xi,\overline{\xi}]_{(wp(\nu))}\xrightarrow{\vert_{\xi=0}}H^\bullet(\mathrm{Fix}(M),\C)[\overline{\xi}]_{(wp^{\overline{\partial}}(\nu))}
\]
one sees it is manifestly surjective. By choosing a linear section $\sigma$ to this map one defines a morphism
\[
\iota_*^{\sigma}\colon {H_{\C/\Lambda;\overline{\partial}}^{\bullet}(\mathrm{Fix}(M),\C)_{(wp^{\overline{\partial}}(\nu))}}[-\mathrm{rk}\nu]\to 
H_{\C/\Lambda;\overline{\partial}}^\bullet(M,\C)_{(wp^{\overline{\partial}}(\nu))}
\]
as the composition
\[
\begin{tikzcd}
{H_{\C/\Lambda}^{\bullet}(\mathrm{Fix}(M),\C)_{(wp(\nu))}}[-\mathrm{rk}\nu]\ar[rr,"\iota_*"]&&H_{\C/\Lambda}^\bullet(M,\C)_{(wp(\nu))}\ar[d,"\vert_{\xi=0}"]
\\
{H_{\C/\Lambda;\overline{\partial}}^{\bullet}(\mathrm{Fix}(M),\C)_{(wp^{\overline{\partial}}(\nu))}}[-\mathrm{rk}\nu]\ar[u,"\sigma"]
&&
H_{\C/\Lambda;\overline{\partial}}^\bullet(M,\C)_{(wp^{\overline{\partial}}(\nu))}
\end{tikzcd}.
\]
The linear map $\iota_*^{\sigma}$ is an inverse to the isomorphism
\[
\begin{tikzcd}
H_{\C/\Lambda;\overline{\partial}}^\bullet(M,\C)_{(wp^{\overline{\partial}}(\nu))} \arrow[rr, "\mathrm{eul}_{\C/\Lambda}^{\overline{\partial}}(\nu)^{-1}\cdot \iota^*", "\sim"'] &  & {H_{\C/\Lambda;\overline{\partial}}^{\bullet}(\mathrm{Fix}(M),\C)_{(wp^{\overline{\partial}}(\nu))}}[-\mathrm{rk}\nu].
\end{tikzcd}
\]
Namely, we have in cohomology
\[
    \mathrm{eul}^{\overline{\partial}}_{\C/\Lambda}(\nu)^{-1}\cdot \iota^* \iota^\sigma_\ast
    =
    \left(\mathrm{eul}_{\C/\Lambda}(\nu)^{-1}\cdot \iota^*\iota_*\sigma \right)\bigr\vert_{\xi=0}=
    \sigma\bigr\vert_{\xi=0}=\mathrm{id}.
\]
By uniqueness of the inverse, this in particular shows that $\iota_*^{\sigma}$ is actually independent of the choice of the section $\sigma$ and so we can unambiguously write $\iota_\ast$ for it. By construction, the morphism
\[
\iota_*\colon {H_{\C/\Lambda;\overline{\partial}}^{\bullet}(\mathrm{Fix}(M),\C)_{(wp^{\overline{\partial}}(\nu))}}[-\mathrm{rk}\nu]\to 
H_{\C/\Lambda;\overline{\partial}}^\bullet(M,\C)_{(wp^{\overline{\partial}}(\nu))}
\]
serves as the pushforward map between the antiholomorphic sectors. 
\end{proof}

\begin{corollary}[Localization Formula in the Antiholomorphic Sector]\label{thm:localization-formula-antihol}
Let $\omega_{\overline{\xi}}$ be a $\overline{\partial}_{\C/\Lambda}$-closed form of degree $\dim M$ in $\Omega^\bullet(M;\C)^{\C/\Lambda}[\overline{\xi}]$. If $\omega_{\overline{\xi}}$ admits a degree $\dim M$ $d_{\C/\Lambda}$-closed extension to $\Omega^\bullet(M;\C)[\xi,\overline{\xi}]$
then
\[
\int_{M}\omega_{\overline{\xi}} =(i\overline{\xi}_\Lambda)^{-\frac{\rk \nu}{2}}
\left(\prod_{\lambda\in \Lambda\setminus\{0\}}\lambda ^{-\frac{\mathrm{rk} (\nu\otimes \C)_{\rho_\lambda}}{2}} \right) \int_{\mathrm{Fix}(M)}\frac{\iota^*\omega_{\overline{\xi}}}{
\widehat{\mathrm{eul}^{\overline{\partial}}_{\C/\Lambda}}(\nu)}.
\]
\end{corollary}
\begin{proof}
Let $\tilde{\omega}(\xi,\overline{\xi})$ be a degree $\dim M$ $d_{\C/\Lambda}$-closed extension of $\omega_{\overline{\xi}}$ to $\Omega^\bullet(M;\C)[\xi,\overline{\xi}]$.
By the localization formula (Corollary \ref{cor:localization-formula}) we have
\[
\int_{M}\omega_{\overline{\xi}} = \int_{\mathrm{Fix}(M)}\frac{\iota^*\widetilde{\omega}(\xi,\overline{\xi})}{
\mathrm{eul}_{\C/\Lambda}(\nu)}.
\]
Since the left hand side is independent of $\xi$, so is the right hand side. Therefore we can write
\[
\int_{M}\omega_{\overline{\xi}} =\left( \int_{\mathrm{Fix}(M)}\frac{\iota^*\widetilde{\omega}(\xi,\overline{\xi})}{
\mathrm{eul}_{\C/\Lambda}(\nu)}\right)\biggr\vert_{\xi=0}.
\]
The rational expression $\mathrm{eul}_{\C/\Lambda}(\nu)^{-1}\cdot \iota^*\widetilde{\omega}(\xi,\overline{\xi})$ is defined at $\xi=0$ and evaluation at $\xi=0$ commutes with equivariant integration (which in the Cartan model is just componentwise integration of the differential form parts). So, by Lemma \ref{lem:anti-holo-eul}, we find
\[
\int_{M} \omega_{\overline{\xi}} = \int_{\mathrm{Fix}(M)}\frac{\iota^*\widetilde{\omega}(\xi,\overline{\xi})}{
\mathrm{eul}_{\C/\Lambda}(\nu)}\biggr\vert_{\xi=0}
=
\int_{\mathrm{Fix}(M)}\frac{\iota^*\omega_{\overline{\xi}}}{
\mathrm{eul}_{\C/\Lambda}(\nu)\vert_{\xi=0}}
=
\int_{\mathrm{Fix}(M)}\frac{\iota^*\omega_{\overline{\xi}}}{
\mathrm{eul}^{\overline{\partial}}_{\C/\Lambda}(\nu)}.
\]

\end{proof}
\begin{example}
Let $M=S^2$ with its standard metric induced by the canonical embedding $S^2\hookrightarrow \R^3$, and let $\omega$ be its volume form. Let us make $U(1)\cong SO(2)$ act on $S^2$ by rotations around the vertical axis, i.e., via the embedding $SO(2)\hookrightarrow SO(3)$ given by $A\mapsto \mathrm{diag}(A,1)$. For any nonzero $\lambda\in \Lambda\subset\C$, use the character $\rho_\lambda\colon \C/\Lambda\to U(1)$ to define a $\C/\Lambda$-action on $S^2$. Using stereographic coordinates on $S^2$ and polar coordinates on $\R^2$, one sees that
\[
\omega_{\overline{\xi}}=\omega-\frac{4\pi}{1+\rho^2}\lambda \overline{\xi}_\Lambda
\]
is a degree 2 $\overline{\partial}_{\C/\Lambda}$-closed form and that 
\[
\omega(\xi,\overline{\xi})=\omega-\frac{4\pi }{1+\rho^2}(\lambda \overline{\xi}_\Lambda-\overline{\lambda} \xi_\Lambda)
\]
is a degree 2 $d_{\C/\Lambda}$-closed extension of $\omega_{\overline{\xi}}$. The $\C/\Lambda$-action on $S^2$ has exactly two fixed points, the North pole corresponding to $\rho=\infty$ and the South pole corresponding to $\rho=0$. Since the manifold of fixed points is 0-dimensional, the normalized equivariant Euler class in the antiholomorphic sector reduces to the constant $1$. Choosing the arguments of $\lambda$ and $-\lambda$ in such a way that $\mathrm{arg}(-\lambda)=\mathrm{arg}(\lambda)-\pi$, the induced orientation on the manifold of fixed points gives positive orientation to the North pole and negative orientation to the South pole. From Corollary \ref{thm:localization-formula-antihol} we then find
\[
\int_{S^2}\omega=\int_{S^2}\omega_{\overline{\xi}}=
(i\overline{\xi}_\Lambda)^{-1}
\lambda^{-1/2}(-\lambda)^{-1/2} \int_{\mathrm{Fix}(S^2)}\left(-\frac{4\pi }{1+\rho^2}\right)\lambda \overline{\xi}_\Lambda=4\pi.
\]
\end{example}

\section{{\color{purple}$\C^\ast$-equivariant families} of $\C/\Lambda$-manifolds and modularity}

So far we have been considering a single $\C/\Lambda$-manifold $M$ {\color{purple}together with a $\C/\Lambda$-equivariant bundle $E$}, for a fixed lattice $\Lambda$. Interesting phenomena happen if we let the lattice, the manifold {\color{purple}and the bundle} vary. More precisely, we will be interested into a smooth family $M_\Lambda$ of $\C/\Lambda$-manifolds {\color{purple}equipped with a smooth family of $\C/\Lambda$-equivariant bundles $E_\Lambda$}, with $\Lambda$ ranging over all oriented lattices in $\mathbb{C}$. {\color{purple}And, in particular, we will be interested in the case when these data are $\mathbb{C}^\ast$-equivariant.}  
\begin{remark}\label{rem:smooth-family}
By saying that the famil{\color{purple}ies we are considering are} smooth we are implicitly saying that the set of all oriented lattices in $\C$ has a natural structure of a smooth manifold. It is indeed so: any lattice $\Lambda\subseteq \C$ admits a basis given by an ordered pair $(\omega_1,\omega_2)\in \C^2$ with $\Im(\overline{\omega}_1\omega_2)>0$. Denoting this open subset of $\C^2$ by $\mathcal{U}$, one has that
\[
\mathrm{Lattices}^+(\C)=\mathcal{U}/SL(2;\mathbb{Z}).
\]
Since the $SL(2;\mathbb{Z})$-action on $\mathcal{U}$ is free and properly discontinuous, one sees that $\mathrm{Lattices}^+(\C)$ is naturally a smooth manifold. The intuitive notion of a smooth family of manifolds parametrized by lattices can then be formalized as the datum of a smooth and proper submersion $\mathcal{M}\to \mathrm{Lattices}^+(\C)$. Similarly, one formalizes the notion of a smooth family of $\C/\Lambda$-manifolds by working in the category of smooth group actions over the base manifold $\mathrm{Lattices}^+(\C)$. {\color{purple} Analogous considerations apply to the definition of smooth families of $\C/\Lambda$-equivariant bundles. However, it turns out that as soon as $\C^\ast$-equivariance comes into play the only possible way of having nontrivial smooth families of $\C/\Lambda$-equivariant bundles is by considering infinite rank bundles. This abruptly takes us out of the finite rank constructions of the previous sections and the product formulas defining equivariant Chern classes become infinite products that need regularization in order to be meaningful. We are going to work this out in detail in the example leading to the Witten genus in Section \ref{sec:Witten}. Here, in consideration of the fact that even a detailed presentation of $\C/\Lambda$-families in terms of the base manifold $\mathrm{Lattices}^+(\C)$ will actually not allow us to derive the main results of this Section in terms of the finite-rank technology developed so far, } we will content ourselves in giving ``fibrewise definitions'', leaving to the interested reader task of writing them in terms of global objects over a base. {\color{purple}Although the definitions will be given fibrewise, it is always assumed that whenever an entity depends on a lattice $\Lambda$ the dependence is smooth in $\Lambda$. More crucially, we will pretend to be unaware that the nontrivial bundles occurring in the constructions in the present Section have necessarily infinite rank, and only at the very end of the Section we will reveal the reader where this infinite-dimensionality occurs. So the whole Section is written pretending that the bundles appearing in it are of finite rank, and has to be intended as a presentation of the abstract nonsense motivating why one should expect the modular properties of the Witten genus. These properties will then be rigorously derived in Section \ref{sec:Witten}. We also remark that, since we are ultimately interested only in manifolds with trivial $\C/\Lambda$-actions, one could give the definitions in this Section directly in this setting, simplifying both the exposition and the construction. However, this shortcut would not be very natural from the point of view of the family $M_\Lambda:=\mathrm{Maps}(\C/\Lambda,X)$ leading to the Witten genus, so we prefer to avoid it. The idea is that, while this Section is only fully rigorous for trivial examples, yet it provides the abstract geometric framework where the construction of the Witten genus and a proof of its modular properties take place and so it can help understanding their specific presentation given in Section \ref{sec:Witten}.}
\end{remark}

The multiplicative group $\C^\ast$ of nonzero complex numbers smoothly acts on the set of lattices by homotheties. The multiplication by a nonzero complex number $a$ induces a complex Lie group isomorphism
\[
m_{a,\Lambda}\colon \C/\Lambda\xrightarrow{\sim} \C/a\Lambda
\]
for any $\Lambda$ in $\mathrm{Lattices}^+(\C)$, and these isomorphisms satisfy
\[
m_{1,\Lambda}=\mathrm{id}_{\C/\Lambda}; \qquad 
m_{a_1a_2,\Lambda}=m_{a_1,a_2\Lambda}\circ m_{a_2,\Lambda}=m_{a_2,a_1\Lambda}\circ m_{a_1,\Lambda}
\]
We say that the family $\{M_\Lambda\}$ of $\C/\Lambda$-manifolds is a {\color{purple}$\C^\ast$-equivariant family} if it is compatible with this $\C^\ast$-action. More precisely, we give the following.
\begin{definition}\label{def:conformal-family}
An \emph{{\color{purple}$\C^\ast$-equivariant family}} of $\C/\Lambda$-manifolds is a smooth family $\{M_\Lambda\}$ of $\C/\Lambda$-manifolds such that, for any $a\in \C^\ast$ and any oriented lattice $\Lambda$ one is given diffeomorphisms 
\[
\varphi_{a,\Lambda}\colon M_{\Lambda}\xrightarrow{\sim} M_{a\Lambda}
\]
such that
\[
\varphi_{1,\Lambda}=\mathrm{id}_{M_\Lambda}; \qquad
\varphi_{a_1a_2,\Lambda}=\varphi_{a_1,a_2\Lambda}\circ \varphi_{a_2,\Lambda}=\varphi_{a_2,a_1\Lambda}\circ \varphi_{a_1,\Lambda}
\]
and the diagram 
\begin{equation}\label{eq:compatible-actions}
\begin{tikzcd}
\C/\Lambda\times M_\Lambda \arrow[r]\arrow[d,"{(m_{a,\Lambda},\varphi_{a,\Lambda})}"'] & M_\Lambda\arrow[d,"\varphi_{a,\Lambda}"]\\
\C/a\Lambda\times M_{a\Lambda} \arrow[r] & M_{a\Lambda}
\end{tikzcd}
\end{equation}
commutes.
\end{definition}
\begin{remark}
In the same spirit of Remark \ref{rem:smooth-family}, one can express the notion of a {\color{purple}$\C^\ast$-equivariant family} in terms of smooth fiber bundles over the moduli stack
\[
\mathcal{M}_{1,1}(\C)=\mathrm{Lattices}^+(\C)/\!/\C^\ast=\mathbb{H}/\!/SL(2,\mathbb{Z})
\]
of elliptic curves over $\C$. Notice that neither the $\C^\ast$-action on  $\mathrm{Lattices}^+(\C)$ nor the $SL(2,\mathbb{Z})$ on the upper complex half-plane $\mathbb{H}=\{\tau\in \mathbb{C}\,:\ \Im(\tau)>0\}$ are free due the fact that the multiplication by $-1$ acts trivially. In terms of elliptic curves this corresponds to the standard involution realizing them as ramified double covers of $\mathbb{P}^1\C$. Additionally, there are points with larger stabilizers, corresponding to {\color{purple} lattices with nontrivial symmetries, e.g., the square lattice $\mathbb{Z}[i]\subseteq \C$.}
\end{remark}

{\color{purple}\begin{remark}
Equipping a manifold $X$ with the trivial $\mathbb{C}/\Lambda$-action for any lattice $\Lambda$ produces a $\mathbb{C}^\ast$-equivariant family of equivariant cohomologies of $X$ parametrized by the elliptic curves $\mathbb{C}/\Lambda$. Arguably, taking the antiholomorphic sectors one obtains a holomorphic vector bundle $\mathbb{E}(X)$ over the moduli space of elliptic curves whose holomorphic sections describe elliptic cohomology classes of $X$ with $\mathbb{C}$-coefficients. More precisely, one should have an elliptic Chern character isomorphism $E\ell\ell(X)\otimes \mathbb{C}\xrightarrow{\sim} \Gamma(\mathcal{M}_{1,1};\mathbb{E}(X))$. This statement, motivated by the analogous one established by Daniel Berwick-Evans in the setting of supersymmetric field theories \cite{berwickevans2013witten,berwickevans2019supersymmetric}, will be hopefully investigated in detail elsewhere.
\end{remark}}

\begin{remark}
It is immediate from Definition \ref{def:conformal-family} that the diffeomorphisms $\varphi_{a,\Lambda}$ induce, by restriction to the fixed points loci, diffeomorphisms
\[
\varphi_{a,\Lambda}\colon \mathrm{Fix}(M_\Lambda)\xrightarrow{\sim} \mathrm{Fix}(M_{a\Lambda}).
\]
\end{remark}

Thanks to the compatibilities between the morphisms $(\varphi_{a,\Lambda},m_{a,\Lambda})$ and the elliptic curve actions in a {\color{purple}$\C^\ast$-equivariant family}, one sees that the pullback morphisms $\varphi_{a,\Lambda}^\ast\colon \Omega^\bullet(M_{a\Lambda};\C)\to \Omega^\bullet(M_\Lambda;\C)$ induce, by restriction to invariant forms, pullback morphisms
\[
\varphi_{a,\Lambda}^\ast\colon \Omega^\bullet(M_{a\Lambda};\C)^{\C/a\Lambda}\to \Omega^\bullet(M_\Lambda;\C)^{\C/\Lambda}.
\]
The complex Lie group isomorphism $m_{a,\Lambda}$ induce, by passing to Lie algebras, complex linear isomorphisms of abelian Lie algebras $dm_{a,\Lambda}\colon \mathfrak{t}_\Lambda\to \mathfrak{t}_{a\Lambda}$. Under the isomorphism $\mathrm{Lie}(\C) \xrightarrow{\sim} \mathfrak{t}_\Lambda$ induced by the projection $\C\to \C/\Lambda$, these linear isomorphisms are just multiplications by the complex number $a$. Complexifying and dualizing we obtain the $\C$-linear automorphism of $\mathrm{Lie}(\C)^\vee\otimes \C$ that acts on the distinguished basis $(dz,d\overline{z})$ as $dz\mapsto a\,dz$ and $d\overline{z}\mapsto \overline{a}\,d\overline{z}$. Therefore,in terms of the distinguished basis $(\xi,\overline{\xi})$ of $ (\mathfrak{t}_\Lambda \otimes_\R \C)^\vee[-2]$ consisting of $dz$ and $d\overline{z}$ placed in degree 2, the $\C$-linear isomorphism
\[
\varphi^\ast_{\mathfrak{t};a,\Lambda}\colon \mathfrak{t}_{a\Lambda}^\vee\otimes \C\xrightarrow{\sim} \mathfrak{t}_{\Lambda}^\vee\otimes\C
\]
induced by $m_a$ is given by 
\[
\varphi^\ast_{\mathfrak{t};a,\Lambda}\colon 
\xi\mapsto a\xi;\qquad 
\varphi^\ast_{\mathfrak{t};a,\Lambda}\colon \overline{\xi}\mapsto \overline{a}\overline{\xi}
\]
\begin{remark}\label{rem:how-xi-lambda-changes}
From
\[
\mathrm{vol}(\C/a\Lambda)=a\overline{a}\, \mathrm{vol}(\C/\Lambda),
\]
one sees that
that in terms of the distinguished basis $(\xi_{a\Lambda},\overline{\xi}_{a\Lambda})$ and $(\xi_{\Lambda},\overline{\xi}_{\Lambda})$ the isomorphism $\varphi^\ast_{\mathfrak{t};a,\Lambda}$ reads
\[
\varphi^\ast_{\mathfrak{t};a,\Lambda}\colon \xi_{a\Lambda}\mapsto \overline{a}^{-1}\xi_\Lambda;\qquad 
\varphi^\ast_{\mathfrak{t};a,\Lambda}\colon \overline{\xi}_{a\Lambda}\mapsto {a}^{-1}\overline{\xi}_\Lambda.
\]
\end{remark}
\begin{lemma}
The data of a {\color{purple}$\C^\ast$-equivariant family} of $\C/\Lambda$-manifolds induce isomorphisms of complexified Cartan tricomplexes
\begin{align*}
\varphi_{a,\Lambda}^*\otimes \varphi^\ast_{\mathfrak{t};a,\Lambda} \colon& 
\Omega^\bullet(M_{a\Lambda};\C)^{\C/a\Lambda} \otimes_\C \mathrm{Sym}(({\mathfrak{t}_{a\Lambda}}^{1,0})^\vee[-2]) \otimes_\C \mathrm{Sym}(({\mathfrak{t}_{a\Lambda}}^{0,1})^\vee[-2])\\ &\xrightarrow{\sim} \Omega^\bullet(M_\Lambda;\C)^{\C/\Lambda} \otimes_\C \mathrm{Sym}(({\mathfrak{t}_\Lambda}^{1,0})^\vee[-2]) \otimes_\C \mathrm{Sym}(({\mathfrak{t}_\Lambda}^{0,1})^\vee[-2]).
\end{align*}
\end{lemma}
\begin{proof}
Since the three differentials are trivial on the generators coming from $\mathfrak{t}_{a\Lambda}^\vee$, we only need to check compatibility with differentials on $\C/a\Lambda$-invariant differential forms on $M_{a\Lambda}$. That is, for an element $\omega\in \Omega^\bullet(M_{a\Lambda};\C)^{\C/a\Lambda}$ we have to check the three identities
\begin{align*}
d_{\mathrm{dR}}(\varphi_{a,\Lambda}^*\omega)&=\varphi_{a,\Lambda}^*(d_{\mathrm{dR}}\omega);\\
\xi \iota_{v^{\Lambda}_{\partial/\partial z}}(\varphi_{a,\Lambda}^*\omega)&=
\varphi^\ast_{\mathfrak{t};a,\Lambda}(\xi)\varphi_{a,\Lambda}^*(\iota_{v^{a\Lambda}_{\partial/\partial z}}\omega)\\
\overline{\xi} \iota_{v^{\Lambda}_{\partial/\partial \overline{z}}}(\varphi_{a,\Lambda}^*\omega)&=
\varphi^\ast_{\mathfrak{t};a,\Lambda}(\overline{\xi})\varphi_{a,\Lambda}^*(\iota_{v^{a\Lambda}_{\partial/\partial \overline{z}}}\omega),
\end{align*}
where $v^{\Lambda}_{\partial/\partial z}$ and $v^{a\Lambda}_{\partial/\partial z}$ are the complex vector fields on $M_\Lambda$ and $M_{a\Lambda}$ corresponding to $\partial/\partial z$ via the differentials of the actions of $\C/\Lambda$ and $C/a\Lambda$, respectively, and similarly for  $v^{\Lambda}_{\partial/\partial \overline{z}}$ and $v^{a\Lambda}_{\partial/\partial \overline{z}}$.
The first identity is obvious. Using $\varphi^\ast_{\mathfrak{t};a,\Lambda}(\xi)=a\xi$, the second identity reduces to
$\iota_{v^{\Lambda}_{\partial/\partial z}}(\varphi_{a,\Lambda}^*\omega)=
a\varphi_{a,\Lambda}^*(\iota_{v^{a\Lambda}_{\partial/\partial z}}\omega)$. By definition of the pullback of differential forms, $\iota_{v^{\Lambda}_{\partial/\partial z}}(\varphi_{a,\Lambda}^*\omega)=\varphi_{a,\Lambda}^*(\iota_{d\varphi_{a,\Lambda}(v^{\Lambda}_{\partial/\partial z})}\omega)$, and so we are reduced to proving the identity $
d\varphi_{a,\Lambda}(v^{\Lambda}_{\partial/\partial z})=v^{a\Lambda}_{a\partial/\partial z}$. Since $a\partial/\partial z=dm_{a,\Lambda}\partial/\partial z$, the identity we have to prove is equivalent to the commutativity of the diagram
\[
\begin{tikzcd}
\mathfrak{t}_\Lambda \arrow[r, "v^\Lambda"]\arrow[d,"dm_{a,\Lambda}"'] & \mathrm{VectorFields}(M_\Lambda)\arrow[d,"d\varphi_{a,\Lambda}"]\\
\mathfrak{t}_{a\Lambda} \arrow[r,"v^{a\Lambda}"] & \mathrm{VectorFields}(M_{a\Lambda}),
\end{tikzcd}
\]
which is immediate from \eqref{eq:compatible-actions}. The proof of the third identity is identical.
\end{proof}
\begin{corollary}
The data of a {\color{purple}$\C^\ast$-equivariant family} of $\C/\Lambda$-manifolds induce isomorphisms between the antiholomorphic sector of the complexified Cartan complex of $M_{a\Lambda}$ and that of $M_\Lambda$, for any oriented lattice $\Lambda$ and every $a\in \C^\ast$.
\end{corollary}

\begin{corollary}\label{cor:vector-bundle}
In a {\color{purple}$\C^\ast$-equivariant family}, the $\C/\Lambda$-equivariant cohomology of $M_\Lambda$ and the  $\C/a\Lambda$-equivariant cohomology of $M_{a\Lambda}$ are canonically isomorphic. The same holds for their (anti-)holomorphic sectors. 
\end{corollary}

\begin{remark}
In global terms, Corollary \ref{cor:vector-bundle} amounts to saying that $H^\bullet_{\C/\Lambda}(M_\Lambda;\C)$ and $H^\bullet_{\C/\Lambda;\overline{\partial}}(M_\Lambda;\C)$ define complex vector bundles over the moduli stack $
\mathcal{M}_{1,1}(\C)$.
\end{remark}

\subsection{Equivariant vector bundles over {\color{purple}$\C^\ast$-equivariant families}}
Given a family $\{M_\Lambda\}$ of $\C/\Lambda$-manifolds, we can consider a family of $\C/\Lambda$-equivariant vector bundles $E_\Lambda$ over the fixed loci $\mathrm{Fix}(M_\Lambda)$. Again, regularity of the family $\{E_\Lambda\}$ can be expressed in terms of a single equivariant vector bundle $\mathcal{E}$ over $\mathrm{Fix}(\mathcal{M})$, where $\mathcal{M}\to \mathrm{Lattices}^+(\C)$ is the smooth and proper submersion from Remark \ref{rem:smooth-family}, but here too we will content us with fibrewise definitions. When $\{M_\Lambda\}$ is a {\color{purple}$\C^\ast$-equivariant family}, it is natural to consider vector bundles $E_\Lambda$ that form a {\color{purple}$\C^\ast$-equivariant family}, too. This leads to the following.

\begin{definition}
A {\color{purple}$\C^\ast$-equivariant family} of $\C/\Lambda$-equivariant vector bundles $E_\Lambda$ over the fixed loci of a {\color{purple}$\C^\ast$-equivariant family} of $\C/\Lambda$-manifolds is a smooth family of equivariant vector bundles equipped with isomorphisms of vector bundles
\[
\psi_{a,\Lambda}\colon E_\Lambda\xrightarrow{\sim} \varphi_{a,\Lambda}^\ast E_{a\Lambda}
\]
making the diagrams 
\[
\begin{tikzcd}
\C/\Lambda\times E_\Lambda \arrow[r]\arrow[d,"{(m_{a,\Lambda},\psi_{a,\Lambda})}"'] & E_\Lambda\arrow[d,"\psi_{a,\Lambda}"]\\
\C/a\Lambda\times \varphi_{a,\Lambda}^\ast E_{a\Lambda} \arrow[r] & \varphi_{a,\Lambda}^\ast E_{a\Lambda}
\end{tikzcd}
\]
commute
for any oriented lattice $\Lambda$ and any $a\in \C^\ast$, such that
\begin{equation}\label{eq:chain-rule}
\psi_{a_1a_2,\Lambda}=\varphi_{a_1,\Lambda}^\ast(\psi_{a_2,a_1\Lambda})\circ \psi_{a_1,\Lambda}
\end{equation}
for any $a_1,a_2$.
\end{definition}
\begin{example}\label{ex:normal-is-conformal}
The restrictions to the fixed loci of the tangent bundles $TM_\Lambda$ for a {\color{purple}$\C^\ast$-equivariant family} $\{M_\Lambda\}$ are a {\color{purple}$\C^\ast$-equivariant family} of $\C/\Lambda$-equivariant vector bundles, with isomorphisms $\psi_{a,\Lambda}$ given by the differentials of the diffeomorphisms $\varphi_{a,\Lambda}$:
\[
d\varphi_{a,\Lambda}\colon TM_\Lambda\xrightarrow{\sim} \varphi_{a,\Lambda}^\ast TM_{a\Lambda}.
\]
Equation \eqref{eq:chain-rule} in this case is
\[
d\varphi_{a_1a_2,\Lambda}=\varphi_{a_1,\Lambda}^\ast(d\varphi_{a_2,a_1\Lambda})\circ d\varphi_{a_1,\Lambda}
\]
and so it is satisfied due to the chain rule for differentials, since $\varphi_{a_1a_2,\Lambda}=\varphi_{a_2,a_1\Lambda}\circ \varphi_{a_1,\Lambda}$. The tangent bundles $T\mathrm{Fix}(M_\Lambda)$ to the fixed loci form a {\color{purple}$\C^\ast$-equivariant family} of subbundles of $\{TM_\Lambda\bigr\vert_{\mathrm{Fix}(M_\Lambda)}\}$, and so the normal bundles to the fixed loci
\[
\nu_\Lambda=\frac{TM_\Lambda\bigr\vert_{\mathrm{Fix}(M_\Lambda)}}{T\mathrm{Fix}(M_\Lambda)}
\]
are a {\color{purple}$\C^\ast$-equivariant family}.
\end{example}
\begin{lemma}\label{lem:pullback-conformal}
Let $\{(L_\Lambda,\chi_\Lambda)\}$ be a {\color{purple}$\C^\ast$-equivariant family} of $\C/\Lambda$-equivariant complex line bundles on $\mathrm{Fix}(M_\Lambda)$. Then
\[
(\varphi_{a,\Lambda}^\ast\otimes\varphi_{\mathfrak{t}a,\Lambda}^\ast)
(c_{1,\C/a\Lambda}(L_{a\Lambda},\chi_{a\Lambda}))=
c_{1,\C/\Lambda}(L_{\Lambda},\chi_{\Lambda})
\]
and
\[
(\varphi_{a,\Lambda}^\ast\otimes\varphi_{\mathfrak{t}a,\Lambda}^\ast)
(c^{\overline{\partial}}_{1,\C/a\Lambda}(L_{a\Lambda},\chi_{a\Lambda}))=
c^{\overline{\partial}}_{1,\C/\Lambda}(L_{\Lambda},\chi_{\Lambda}).
\]
\end{lemma}
\begin{proof}
In the notation of Remark \ref{rem:notation-rho}, the character $\chi_\Lambda$ will be of the form $\rho_\lambda$ for some $\lambda\in \Lambda$. By definition of {\color{purple}$\C^\ast$-equivariant family}, the diagram
\[
\begin{tikzcd}
\C/\Lambda\times L_\Lambda \arrow[r,"\chi_\Lambda\cdot"]\arrow[d,"{(m_{a,\Lambda},\psi_{a,\Lambda})}"'] & L_\Lambda\arrow[d,"\psi_{a,\Lambda}"]\\
\C/a\Lambda\times \varphi_{a\Lambda}^\ast L_{a\Lambda} \arrow[r,"\chi_{a\Lambda}\cdot"] & \varphi_{a\Lambda}^\ast L_{a\Lambda}
\end{tikzcd}
\]
commutes, and so $\chi_\Lambda(z)= \chi_{a\Lambda}(az)$ for any $z\in \C$. This means that $\chi_{a\Lambda}=\rho_{a\lambda}$. Now we compute
\[
c_{1,\C/a\Lambda}(L_{a\Lambda},\chi_{a\Lambda})=
c_{1,\C/a\Lambda}(L_{a\Lambda},\rho_{a\lambda}) = c_1(L_{a\Lambda}) + a\lambda\overline{\xi}_{a\Lambda}-\overline{a\lambda}\xi_{a\Lambda},
\]
and so by Remark \ref{rem:how-xi-lambda-changes} we have
\begin{align*}
    (\varphi_{a,\Lambda}^\ast\otimes\varphi_{\mathfrak{t}a,\Lambda}^\ast)
(c_{1,\C/a\Lambda}(L_{a\Lambda},\chi_{a\Lambda}))&=
(\varphi_{a,\Lambda}^\ast\otimes\varphi_{\mathfrak{t}a,\Lambda}^\ast)
(c_1(L_{a\Lambda}) + a\lambda\overline{\xi}_{a\Lambda}-\overline{a\lambda}\xi_{a\Lambda})\\
&=
c_1(\varphi_{a,\Lambda}^\ast L_{a\Lambda})+\lambda\overline{\xi}_{\Lambda}-\overline{\lambda}\xi_{\Lambda}\\
&=c_1(L_{\Lambda})+\lambda\overline{\xi}_{\Lambda}-\overline{\lambda}\xi_{\Lambda}\\
&=c_{1,\C/\Lambda}(L_{\Lambda},\chi_{\Lambda})).
\end{align*}
The proof for the antiholomorphic part is analogous. 
\end{proof}
By the splitting principle we therefore get the following.
\begin{proposition}\label{prop:invariance-top-chern}
Let $\{E_\Lambda\}$ be a {\color{purple}$\C^\ast$-equivariant family} of $\C/\Lambda$-equivariant complex vector bundles over $\mathrm{Fix}(M_\Lambda)$. Then we have
\[
(\varphi_{a,\Lambda}^\ast\otimes\varphi_{\mathfrak{t};a,\Lambda}^\ast)
(c^{\overline{\partial}}_{\mathrm{top},\C/a\Lambda}(E_{a\Lambda}^\mathrm{eff}))=
c^{\overline{\partial}}_{\mathrm{top},\C/\Lambda}(E_\Lambda^\mathrm{eff})
\]
and
\[
(\varphi_{a,\Lambda}^\ast\otimes\varphi_{\mathfrak{t};a,\Lambda}^\ast)
(\widehat{c^{\overline{\partial}}_{\mathrm{top},\C/a\Lambda}}(E_{a\Lambda}^\mathrm{eff}))=
\widehat{c^{\overline{\partial}}_{\mathrm{top},\C/\Lambda}}(E_\Lambda^\mathrm{eff}).
\]
\end{proposition}
\begin{proof}
The first identity is immediate from Lemma \ref{lem:pullback-conformal} and the splitting principle. To prove the second identity, we write
\[
c^{\overline{\partial}}_{\mathrm{top},\C/a\Lambda}(E_{a\Lambda}^\mathrm{eff}) 
={\overline{\xi}_{a\Lambda}}^{\rk E_{a\Lambda}^\mathrm{eff}}\left(\prod_{a\lambda\in a\Lambda\setminus\{0\}}(a\lambda)^{\mathrm{rk} E_{a\Lambda;\rho_{a\lambda}}} \right) \widehat{c_{\mathrm{top},\C/\Lambda}^{\overline{\partial}}}(E_{a\Lambda}^\mathrm{eff}),
\]
notice that $\mathrm{rk} E_{a\Lambda;\rho_{a\lambda}}=\mathrm{rk} E_{\Lambda;\rho_{\lambda}}$for every $\lambda\in \Lambda$, and use the first identity to get
\begin{align*}
    {\overline{\xi}_\Lambda}^{\rk E_\Lambda^\mathrm{eff}}&\left(\prod_{\lambda\in \Lambda\setminus\{0\}}\lambda ^{\mathrm{rk} E_{\Lambda;\rho_\lambda}} \right) \widehat{c_{\mathrm{top},\C/\Lambda}^{\overline{\partial}}}(E_\Lambda^\mathrm{eff})=c^{\overline{\partial}}_{\mathrm{top},\C/\Lambda}(E_\Lambda^\mathrm{eff})\\
    &=(\varphi_{a,\Lambda}^\ast\otimes\varphi_{\mathfrak{t};a,\Lambda}^\ast)
({c^{\overline{\partial}}_{\mathrm{top},\C/a\Lambda}(E_{a\Lambda}^\mathrm{eff})})\\
&=(\varphi_{\mathfrak{t};a,\Lambda}^\ast{\overline{\xi}_{a\Lambda}})^{\rk E_{\Lambda}^\mathrm{eff}}\left(\prod_{\lambda\in \Lambda\setminus\{0\}}(a\lambda)^{\mathrm{rk} E_{\Lambda;\rho_{\lambda}}} \right)(\varphi_{a,\Lambda}^\ast\otimes\varphi_{\mathfrak{t};a,\Lambda}^\ast) \widehat{c_{\mathrm{top},\C/a\Lambda}^{\overline{\partial}}}(E_{a\Lambda}^\mathrm{eff})\\
&=\overline{\xi}_{\Lambda}^{\rk E_{\Lambda}^\mathrm{eff}}\left(\prod_{\lambda\in \Lambda\setminus\{0\}}\lambda^{\mathrm{rk} E_{\Lambda;\rho_{\lambda}}} \right)(\varphi_{a,\Lambda}^\ast\otimes\varphi_{\mathfrak{t};a,\Lambda}^\ast) \widehat{c_{\mathrm{top},\C/a\Lambda}^{\overline{\partial}}}(E_{a\Lambda}^\mathrm{eff}),
\end{align*}
and the conclusion follows.
\end{proof}
\begin{corollary}\label{cor:identical}
Let $\{V_\Lambda\}$ be a {\color{purple}$\C^\ast$-equivariant family} of $\C/\Lambda$-equivariant real vector bundles on $\mathrm{Fix}(M_\Lambda)$. Then we have
\[
(\varphi_{a,\Lambda}^\ast\otimes\varphi_{\mathfrak{t};a,\Lambda}^\ast)(\widehat{\mathrm{eul}^{\overline{\partial}}_{\C/a\Lambda}}(V_{a\Lambda}^\mathrm{eff}))=\widehat{\mathrm{eul}^{\overline{\partial}}_{\C/\Lambda}}(V_\Lambda^\mathrm{eff})
\]
 \end{corollary}
Specializing this to the case considered in Example \ref{ex:normal-is-conformal} we find the following.
\begin{corollary}
Let $\{M_\Lambda\}$ be a {\color{purple}$\C^\ast$-equivariant family} of $\C/\Lambda$-manifolds, and let $\nu_\Lambda$ be the normal bundle for $\mathrm{Fix}(M_\Lambda)\hookrightarrow M_\Lambda$. Then we have
\[
(\varphi_{a,\Lambda}^\ast\otimes\varphi_{\mathfrak{t};a,\Lambda}^\ast)(\widehat{\mathrm{eul}^{\overline{\partial}}_{\C/a\Lambda}}(\nu_{a\Lambda}))=\widehat{\mathrm{eul}^{\overline{\partial}}_{\C/\Lambda}}(\nu_\Lambda)
\]
for any oriented lattice $\Lambda$ and any $a\in \C^\ast$.
\end{corollary}
\subsection{Trivializations of the fixed points bundle and modular forms}\label{sec:fixed-modular}
Assume now we have a trivialization of the family of fixed points submanifolds of a {\color{purple}$\C^\ast$-equivariant family} $\{M_\Lambda\}$. In terms of smooth bundles over the moduli stack $\mathcal{M}_{1,1}$, this is a trivialization of the smooth fiber bundle $\mathrm{Fix}(\mathcal M)\to \mathcal{M}_{1,1}$. In terms of fibrewise definitions, this is the following.
\begin{definition}\label{def:trivialization}
Let $\{M_\Lambda\}$ be a {\color{purple}$\C^\ast$-equivariant family}. A trivialization of the {\color{purple}$\C^\ast$-equivariant family} $\{\mathrm{Fix}(M_\Lambda)\}$ is the datum of a smooth manifold $X$ and of a collection of diffeomorphisms 
\[
j_\Lambda\colon X\xrightarrow{\sim} \mathrm{Fix}(M_\Lambda)
\]
such that the diagrams
\begin{equation}\label{eq:commute-trivial}
\begin{tikzcd}
 & \mathrm{Fix}(M_\Lambda)\arrow[dd,"\varphi_{a,\Lambda}"]\\
X  \arrow[ru,"j_{\Lambda}"]\arrow[rd,"j_{a\Lambda}"']&\\
& \mathrm{Fix}(M_{a\Lambda})
\end{tikzcd}
\end{equation}
commute for any oriented lattice $\Lambda$ and any $a\in \C^\ast$.
\end{definition}
A trivialization of the fixed points bundle produces a trivialization of the $\C/\Lambda$-equivariant cohomologies of the fixed loci and of their antiholomorphic sectors. More precisely, we have the following lemma, whose proof is immediate from Remark \ref{rem:how-xi-lambda-changes} and Definition \ref{def:trivialization}. 
\begin{lemma}\label{lem:trivialization}
Let $\xi_X$ and $\overline{\xi}_X$ be two variables in degree 2, and let $j_{\mathfrak{t};\Lambda}^*\colon \C[\xi_\Lambda,\overline{\xi}_\Lambda]\xrightarrow{\sim} \C[\xi_X,\overline{\xi}_X]$ be the ring isomorphism induced by 
\[
j_{\mathfrak{t};\Lambda}^*\colon \xi_\Lambda\mapsto \xi_X; \quad j_{\mathfrak{t};\Lambda}^*\colon \overline{\xi}_\Lambda\mapsto \overline{\xi}_X.
\]
For any $a\in \C^\ast$, let $\mu_a\colon \C[\xi_X,\overline{\xi}_X]\xrightarrow{\sim} \C[\xi_X,\overline{\xi}_X]$ be the ring isomorphism induced by 
\[
\mu_a\colon \xi_X\mapsto \overline{a}^{-1}\xi_X; \quad 
\mu_a\colon \overline{\xi}_X\mapsto a^{-1}\overline{\xi}_X.
\]
Then 
\[
H^\bullet_{\C/\Lambda}(\mathrm{Fix}(M_\Lambda),\C)\cong 
H^\bullet(\mathrm{Fix}(M_\Lambda),\C)[\xi_\Lambda,\overline{\xi}_\Lambda]\xrightarrow{j_\Lambda^\ast\otimes j_{\mathfrak{t};\Lambda}^*}
H^\bullet(X,\C)[\xi_X,\overline{\xi}_X]
\]
is an isomorphism of graded rings, and all the diagrams
\[
\begin{tikzcd}
  H^\bullet_{\C/a\Lambda}(\mathrm{Fix}(M_{a\Lambda}),\C)\arrow[d,"\varphi_{a,\Lambda}^\ast\otimes \varphi_{\mathfrak{t};a,\Lambda}^\ast"'] \arrow[rr,"j_{a\Lambda}^\ast\otimes j_{\mathfrak{t};a\Lambda}^*"]&& H^\bullet(X,\C)[\xi_X,\overline{\xi}_X]\arrow[d,"\mathrm{id}\otimes \mu_a"]
  \\
 H^\bullet_{\C/\Lambda}(\mathrm{Fix}(M_\Lambda),\C)\arrow[rr,"j_\Lambda^\ast\otimes j_{\mathfrak{t};\Lambda}^*"']&& H^\bullet(X,\C)[\xi_X,\overline{\xi}_X]
\end{tikzcd}
\]
commute. The same hold for the antiholomorphic sectors: 
\[
H^\bullet_{\C/\Lambda;\overline{\partial}}(\mathrm{Fix}(M_\Lambda),\C)\cong 
H^\bullet(\mathrm{Fix}(M_\Lambda),\C)[\overline{\xi}_\Lambda]\xrightarrow{j_\Lambda^\ast\otimes j_{\mathfrak{t};\Lambda}^*}
H^\bullet(X,\C)[\overline{\xi}_X]
\]
is an isomorphism of graded rings, and all the diagrams
\[
\begin{tikzcd}
  H^\bullet_{\C/a\Lambda;\overline{\partial}}(\mathrm{Fix}(M_{a\Lambda}),\C)\arrow[d,"\varphi_{a,\Lambda}^\ast\otimes \varphi_{\mathfrak{t};a,\Lambda}^\ast"'] \arrow[rr,"j_{a\Lambda}^\ast\otimes j_{\mathfrak{t};a\Lambda}^*"]&& H^\bullet(X,\C)[\overline{\xi}_X]\arrow[d,"\mathrm{id}\otimes \mu_a"]
  \\
 H^\bullet_{\C/\Lambda;\overline{\partial}}(\mathrm{Fix}(M_\Lambda),\C)\arrow[rr,"j_\Lambda^\ast\otimes j_{\mathfrak{t};\Lambda}^*"']&& H^\bullet(X,\C)[\overline{\xi}_X]
\end{tikzcd}
\]
commute. 
\end{lemma}

The second statement in Lemma \ref{lem:trivialization} clearly continues to hold after we localize $H^\bullet(\mathrm{Fix}(M_\Lambda),\C)[\overline{\xi}_\Lambda]$ at $\overline{\xi}_\Lambda$ and $H^\bullet(X,\C)[\overline{\xi}_X]$ at $\overline{\xi}_X$.
Recalling Remark \ref{rem:even-powers}, we can now give the main definition of this section.
\begin{definition}\label{def:witten-classes}
Let $\{M_\Lambda\}$ be a {\color{purple}$\C^\ast$-equivariant family} of $\C/\Lambda$-manifolds with a given trivialization $(X,\{j_\Lambda\})$ of its fixed points, and let $\mathcal{E}=\{E_\Lambda\}$ and $\mathcal{V}=\{V_\Lambda\}$ be {\color{purple}$\C^\ast$-equivariant families} of $\C/\Lambda$-equivariant complex and real vector bundles over $\mathrm{Fix}(M_\Lambda)$, respectively. The \emph{total complex {\color{purple}$\mathcal{W}$-}class}
of $\mathcal{E}$ is the function
\[
{\color{purple}\mathcal{W}}_\mathcal{E}\colon\mathrm{Lattices}^+(\C)\to H^\bullet(X,\C)[\overline{\xi}_X^{-1}]
\]
given by
\[
{\color{purple}\mathcal{W}}_\mathcal{E}(\Lambda):=(j_\Lambda^\ast\otimes j_{\mathfrak{t};\Lambda}^*)\left(\frac{1}{\widehat{c^{\overline{\partial}}_{\mathrm{top},\C/\Lambda}}(E_\Lambda^\mathrm{eff})}\right).
\]
The coefficient
\[
{\color{purple}\mathcal{W}}_{\mathcal{E};k}\colon \mathrm{Lattices}^+(\C)\to H^{2k}(X,\C)
\]
defined by the expansion 
\[
{\color{purple}\mathcal{W}}_\mathcal{E}(\Lambda)=\sum_{k=0}^\infty {\color{purple}\mathcal{W}}_{\mathcal{E};k}(\Lambda)\overline{\xi}_X^{-k}
\]
will be called the $k$-th complex {\color{purple}$\mathcal{W}$-}class of $\mathcal{E}$.
The \emph{total real {\color{purple}$\mathcal{W}$-}class} of $\mathcal{V}$ is the function
\[
{\color{purple}\mathcal{W}}_{\mathbb{R};\mathcal{V}}\colon\mathrm{Lattices}^+(\C)\to H^\bullet(X,\C)[\overline{\xi}_X^{-1}]
\]
given by
\[
{\color{purple}\mathcal{W}}_{\mathbb{R};\mathcal{V}}(\Lambda):=(j_\Lambda^\ast\otimes j_{\mathfrak{t};\Lambda}^*)\left(\frac{1}{\widehat{\mathrm{eul}^{\overline{\partial}}_{\C/\Lambda}}(V_\Lambda^\mathrm{eff})}\right).
\]
The coefficient
\[
{\color{purple}\mathcal{W}}_{\mathbb{R};\mathcal{V};k}\colon \mathrm{Lattices}^+(\C)\to H^{4k}(X,\C)
\]
defined by the expansion 
\[
{\color{purple}\mathcal{W}}_{\mathbb{R};\mathcal{V};k}(\Lambda)=\sum_{k=0}^\infty {\color{purple}\mathcal{W}}_{\mathbb{R};\mathcal{V};k}(\Lambda)\overline{\xi}_X^{-2k}
\]
will be called the $k$-th real {\color{purple}$\mathcal{W}$-}class of $\mathcal{V}$.
\end{definition}
\begin{proposition}\label{prop:modularity}
The $k$-th complex {\color{purple}$\mathcal{W}$-}class ${\color{purple}\mathcal{W}}_{\mathcal{E};k}$ is a modular form of weight $k$ {\color{purple}with values in $H^{2k}(X,\C)$}. The $k$-th real {\color{purple}$\mathcal{W}$-}class ${\color{purple}\mathcal{W}}_{\mathbb{R};\mathcal{V};k}$ is a modular form of weight $2k$ with values in {\color{purple}$H^{4k}(X,\C)$}.
\end{proposition}
\begin{proof}
We give a proof for the complex {\color{purple}$\mathcal{W}$-}classes first.
We have to show that for any $a\in \C^\ast$ and any oriented lattice $\Lambda$ we have 
\[
{\color{purple}\mathcal{W}}_{\mathcal{E};k}(a\Lambda)=a^{-k}{\color{purple}\mathcal{W}}_{\mathcal{E},k}(\Lambda)
\]
and that, denoting by $\Lambda_\tau$ the lattice $\Lambda_\tau:=\mathbb{Z}\oplus\tau\mathbb{Z}$ for a complex number $\tau$ in the upper complex half-plane $\mathbb{H}$ one has that
\begin{align*}
 \mathbb{H}&\to H^{2k}(X,\C)[\overline{\xi}_X^{-1}] \\
 \tau&\mapsto {\color{purple}\mathcal{W}}_{\mathcal{E};k}(\Lambda_\tau)
\end{align*}
is a holomorphic function of $\tau$. The first identity follows from Lemma \ref{lem:trivialization} and Proposition \ref{prop:invariance-top-chern}. Indeed, we have
\begin{align*}
    {\color{purple}\mathcal{W}}_\mathcal{E}(a\Lambda)&=(j_{a\Lambda}^\ast\otimes j_{\mathfrak{t};a\Lambda}^*)\left(\frac{1}{\widehat{c^{\overline{\partial}}_{\mathrm{top},\C/a\Lambda}}(E_{a\Lambda}^\mathrm{eff})}\right)\\
    &=(\mathrm{id}\otimes \mu_{a^{-1}})\circ (j_\Lambda^\ast\otimes j^\ast_{\mathfrak{t};\Lambda})\circ(\varphi_{a\Lambda}^\ast\otimes \varphi_{\mathfrak{t};a\Lambda}^*)\left(\frac{1}{\widehat{c^{\overline{\partial}}_{\mathrm{top},\C/a\Lambda}}(E_{a\Lambda}^\mathrm{eff})}\right)\\
    &=(\mathrm{id}\otimes \mu_{a^{-1}})\circ (j_\Lambda^\ast\otimes j^\ast_{\mathfrak{t};\Lambda})\left(\frac{1}{\widehat{c^{\overline{\partial}}_{\mathrm{top},\C/\Lambda}}(E_{\Lambda}^\mathrm{eff})}\right)\\
    &=(\mathrm{id}\otimes \mu_{a^{-1}}){\color{purple}\mathcal{W}}_\mathcal{E}(\Lambda).
\end{align*}
Since $\mu_{a^{-1}}(\overline{\xi}_X^{-1})=a^{-1}\overline{\xi}_X^{-1}$, expanding this identity gives
\[
\sum_{k=0}^\infty {\color{purple}\mathcal{W}}_{\mathcal{E};k}(a\Lambda)\overline{\xi}_X^{-k}=
\sum_{k=0}^\infty a^{-k}{\color{purple}\mathcal{W}}_{\mathcal{E};k}(\Lambda)\overline{\xi}_X^{-k}.
\]
Holomorphicity of $\tau\mapsto {\color{purple}\mathcal{W}}_{\mathcal{E};k}(\Lambda_\tau)$ is immediate from Remark \ref{rem:why-antiholo}. {\color{purple}This shows that $\mathcal{W}_{\mathcal{E};k}$ is a weak modular form of weight $k$. The analysis of the behavior of $\mathcal{W}_{\mathcal{E};k}$ for $\tau\to +i\infty$, showing that $\mathcal{W}_{\mathcal{E};k}$ is indeed a modular form, immediately follows from the discussion at the end of this Section.} The proof for the real {\color{purple}$\mathcal{W}$} classes is identical, by using Corollary \ref{cor:identical}.
\end{proof}

\begin{corollary}\label{cor:modularity}
In the same assumptions as in Definition \ref{def:witten-classes}, if  $X$ is a manifold of even dimension $d$, then
\[
\int_X {\color{purple}\mathcal{W}}_{\mathcal{E};d/2}
\]
is a complex valued modular form of weight $d/2$. If  $X$ is a manifold of dimension $d$ with $d\equiv 0\mod 4$, then
\[
\int_X {\color{purple}\mathcal{W}}_{\mathbb{R};\mathcal{V};d/4}
\]
is a complex valued modular form of weight $d/2$.
\end{corollary}
{\color{purple}
As in every respectable detective story it is now time to unveil the mystery around the somehow suspicious Proposition \ref{prop:modularity} and the role infinite rank bundles secretly play in it: after all, how is it possible that the finite product defining the total complex {\color{purple}$\mathcal{W}$-}class of $\mathcal{E}$ defines a modular form? The point is that Proposition \ref{prop:modularity} is true, but it is trivially true: a $\mathbb{C}^*$-equivariant family of $\C/\Lambda$-equivariant finite rank bundles only carries the trivial action by $\C/\Lambda$, so that $E^\mathrm{eff}_\Lambda=\mathbf{0}$ and the normalized top Chern class of $E^\mathrm{eff}_\Lambda$ is $1$. The $k$-th complex {\color{purple}$\mathcal{W}$} class of $\mathcal{E}$ is then $1$ for $k=0$ and $0$ for $k>0$. These are clearly modular forms, but trivially so. The reason for the triviality of the $\C/\Lambda$-action on $E_\Lambda$ is hidden in Lemma \ref{lem:pullback-conformal}. There, we consider a $\C^\ast$-equivariant family $\{(L_\Lambda,\chi_\Lambda)\}$ of $\C/\Lambda$-equivariant complex line bundles on $\mathrm{Fix}(M_\Lambda)$. By looking at the second component of the pair $(L_\Lambda,\chi_\Lambda)$ we have in particular an association $\Lambda\mapsto \chi_\Lambda$. Identifying the characters of $\C/\Lambda$ with the lattice $\Lambda$ itself, we therefore see we are picking an element of $\Lambda$ or every $\Lambda$. And the only way of doing this in a continuous and equivariant manner is to pick the element $\mathbf{0}$ for any $\Lambda$. Indeed, a continuous $\C^\ast$-equivariant choice of $\chi_\Lambda$ is equivalent to the datum of a continuous and $SL(2;\mathbb{Z})$-equivariant function $\mathbb{H}\to \mathbb{Z}^2$, where one has the standard action of $SL(2;\mathbb{Z})$ on $\mathbb{Z}^2$. Since $\mathbb{H}$ is connected, this is the datum of an $SL(2;\mathbb{Z})$-invariant element of $\mathbb{Z}^2$ and the only such an element is $(0,0)$. If instead of a family of line bundles we consider a family of finite-rank vector bundles not much changes: now for each lattice $\Lambda$ we are picking the set of characters of $\C/\Lambda$ corresponding to the formal decomposition of $E_\Lambda$ into a direct sum of $\C/\Lambda$-equivariant line bundles, and doing this in a continuous way is equivalent to picking a finite $SL(2;\mathbb{Z})$-invariant subset of $\mathbb{Z}^2$. But the only such subset is $\{(0,0)\}$: an invariant subset is an union of orbits, and all the $SL(2;\mathbb{Z})$-orbits in $\mathbb{Z}^2$, except $\{(0,0)\}$ are infinite. Things drastically change if we allow infinite-rank vector bundles. Now we can have nontrivial $SL(2;\mathbb{Z})$-invariant subsets of $\mathbb{Z}^2$ and so the construction in Proposition \ref{prop:modularity} will formally produce a modular form. Yet, since the normalized top Chern class of $E^\mathrm{eff}_\Lambda$ will now be an infinite product with no reason to converge, this modular form will only be a formal expression. What one can do is to try to use regularization techniques to change the possibly divergent infinite product into a convergent one and hope that the modular properties of the formal expression survive the regularization process. In the next Section we show that this is indeed so in the particular case of this general construction leading to the Witten genus of a rational string manifold.
}

\section{The Witten class of double loop spaces}\label{sec:Witten}

Let now $X$ be a smooth $d$-dimensional manifold.
A paradigmatic example of a {\color{purple}$\C^\ast$-equivariant family} of $\C/\Lambda$-manifolds is given
by
\[
M_\Lambda:=\mathrm{Maps}(\C/\Lambda,X),
\]
the spaces of smooth maps from $\C/\Lambda$ to $X$, with their standard Fréchet infinite-dimensional smooth manifold structures, and with $\C/\Lambda$-actions given by translation: $z\star \gamma\colon p\mapsto \gamma(p+z)$. The isomorphisms
\[
\varphi_{a,\Lambda}\colon \mathrm{Maps}(\C/\Lambda,X)\xrightarrow{\sim} \mathrm{Maps}(\C/a\Lambda,X)
\]
are given by the pullbacks along $m_{a^{-1},a\Lambda}\colon \C/a\Lambda\to \C/\Lambda$. The commutativity of \eqref{eq:compatible-actions} is then the trivial identity $a^{-1}(p+az)=a^{-1}p+z$.
The submanifold of fixed points for this action consists of the submanifold of constant loops, so that we have a canonical trivialization $j_\Lambda\colon X\xrightarrow{\sim} \mathrm{Fix}(M_\Lambda)$ mapping a point $x\in X$ to the constant map $\gamma_x\colon \C/\Lambda \to X$ with constant value $x$. The commutativity of \eqref{eq:commute-trivial} is trivial. We are thus in the situation considered in Section \ref{sec:fixed-modular} and so, by 
Proposition \ref{prop:modularity} and Corollary \ref{cor:modularity}, modular forms are naturally associated with the {\color{purple}$\C^\ast$-equivariant family} $\{\mathrm{Maps}(\C/\Lambda,X)\}$ and so, ultimately, to $X$.
\par
Things are however not so straightforward. Indeed, as the normal bundle to $X$ in $\mathrm{Maps}(\C/\Lambda,X)$ has infinite rank, we will need to make sense of the now infinite products defining normalized  equivariant top Chern classes and Euler classes. The idea here is to write
\[
\left(1+\frac{\alpha_{i}(E_{\rho_{\lambda}})\overline{\xi}_\Lambda^{-1}}{\lambda}\right)=\left(1+\frac{z}{\lambda}\right)\biggr\vert_{z=\alpha_{i}(E_{\rho_{\lambda}})\overline{\xi}_\Lambda^{-1}}
\]
where we think of the degree zero variable $z$ as of a complex variable, to compute the product of the factors $1+z/\lambda$ so to obtain an entire function $\Phi(z)$ and then to compute $\Phi(\alpha_{i}(E_{\rho_{\lambda}})\overline{\xi}_\Lambda^{-1})$ by putting $z=\alpha_{i}(E_{\rho_{\lambda}})\overline{\xi}_\Lambda^{-1}$ in the Taylor expansion of $\Phi$ at $z=0$. Notice that this last operation makes sense without any convergence issue as, by the finite dimensionality of $X$, the degree zero element $\alpha_{i}(E_{\rho_{\lambda}})\overline{\xi}_\Lambda^{-1}$ is nilpotent. Yet, there is no guarantee that the infinite product of the factors $1+z/\lambda$ will converge, and actually it does not. So one has to suitably regularize it in order to get a convergent product. A convenient way of doing so is by the technique of Weierstra{\ss} $\zeta$-regularization that we recall below. Our problems are not over, yet: $\zeta$-regularization may disrupt the expected modularity of {\color{purple}$\mathcal{W}$-}classes, so we will have to check in the end whether this is preserved. It will turn out that a topological constraint on $X$ has to be imposed in order to maintain modularity: $X$ has to be a rational string manifold.

\par
With these premises, we can now determine the {\color{purple}$\mathcal{W}$-}classes of the {\color{purple}$\C^\ast$-equivariant family} $\{M_\Lambda\}=\{\mathrm{Maps}(\C/\Lambda,X)\}$.
We will assume $X$ to be $2$-connected so that $M_\Lambda$ is connected. One can weaken this assumption by requiring $X$ to be only connected, and taking $M_\Lambda$ to be the space $\mathrm{Maps}_0(\C/\Lambda,X)$ of homotopically trivial maps from $\C/\Lambda$ to $X$. Or one can even remove any connectedness assumption on $X$ by working separately on each of the connected components of $\mathrm{Maps}_0(\C/\Lambda,X)$ (which bijectively correspond to the connected components of $X$).

As the smooth structure on $M_\Lambda$ is the standard Fr\'echet one, the tangent space at the point $\gamma\in \mathrm{Maps}(\C/\Lambda,X)$ is the space $H^0(\C/\Lambda;\gamma^\ast TX)$ of smooth sections of the pullback of the tangent bundle of $X$ via $\gamma$. In particular, for any $x\in X$ we have
\[
T_{\gamma_x}M_\Lambda=C^\infty(\C/\Lambda;T_xX)= C^\infty(\C/\Lambda;\R)\otimes_\R T_xX,
\]
and so the restriction of the complexified tangent bundle of $M_\Lambda$ to $X=\mathrm{Fix}(M_\Lambda)$ is 
\[
TM_\Lambda\otimes\C\bigr\vert_X=C^\infty(\C/\Lambda;\C)\otimes_\C (TX\otimes_\R \C).
\]
\begin{remark}\label{rem:j-is-identity}
In writing $X=\mathrm{Fix}(M_\Lambda)$ we have identified $X$ with $\mathrm{Fix}(M_\Lambda)$ via $j_\Lambda$. We will keep this identification fixed in all that follows, so that we will see $X$ as a submanifold of $M_\Lambda$ and consequently reduce ${\color{purple}j}_\Lambda$ to the identity of $X$.
\end{remark}
Since Fourier polynomials are dense in the Fr\'echet topology of $C^\infty(\C/\Lambda;\C)$, in the same vein of
\cite{atiyahcircsym}, we consider the 
{\color{purple} dense inclusion} 
\[
TM_\Lambda\otimes\C\bigr\vert_X{\color{purple}\overset{\mathrm{dense}}{\supseteq}}\left(\bigoplus_{\lambda\in \Lambda}\C_{(\lambda)}\right)\otimes_\C (TX\otimes_\R \C),
\]
where $\C_{(\lambda)}$ is the 1-dimensional representation of $\C/\Lambda$ with character $\rho_\lambda$. This immediately implies
\[
\nu_\Lambda\otimes\C\,{\color{purple}\overset{\mathrm{dense}}{\supseteq}}\left(\bigoplus_{\lambda\in \Lambda\setminus\{0\}}\C_{(\lambda)}\right)\otimes_\C (TX\otimes_\R \C),
\]
where $\nu_\Lambda$ denotes the normal bundle for the inclusion $X\hookrightarrow \mathrm{Maps}(\C/\Lambda,X)$.

By formally applying formula (\ref{eq:c-top-with-zeta-antiholo}) to this infinite rank situation, we obtain
\begin{align}
\notag
\widehat{c_{\mathrm{top},\C/\Lambda}^{\overline{\partial}}}(\nu_\Lambda\otimes \C)&=
\prod_{\lambda\in \Lambda\setminus\{0\}}\prod_{i=1}^{d}\left(1+\frac{\alpha_{i}(X)\overline{\xi}_\Lambda^{-1}}{\lambda}\right)\\
\label{eq:chern-weierstrass}
&=\prod_{i=1}^{d}
\prod_{\lambda\in \Lambda\setminus\{0\}}\left(1+\frac{z}{\lambda}\right)\biggr\vert_{z=\alpha_{i}(X)\overline{\xi}_\Lambda^{-1}},
\end{align}
where $\alpha_1(X),\dots,\alpha_d(X)$ are the Chern roots of the complexified tangent bundle $TX\otimes_\R\C$ of $X$. As anticipated, to compute (and actually give a meaning to) the infinite product 
\[
\prod_{\lambda\in \Lambda\setminus\{0\}}\left(1+\frac{z}{\lambda}\right)
\]
one uses Weierstra{\ss} $\zeta$-regularization. For the reader's convenience we recall the basics of the procedure here. A detailed treatment can be found, e.g., in \cite[Chapter 15]{rudin}.
For any $r\geq 0$, let
\[
P_r(z):=\sum_{j=1}^r \frac{z^j}{j}=\begin{cases}
0 &\text{if $r=0$}\\
z+\frac{z^2}{2}+\cdots +\frac{z^r}{r}&\text{if $r>0$}.
\end{cases}
\]
If $\{\kappa_n\}$ is a sequence of nonzero complex numbers with $|\kappa_n|\to +\infty$ for $n\to +\infty$ and $\{p_n\}$ is a sequence of nonnegative integers such that
\[
\sum_{n=1}^{\infty}\left( \frac{R}{|\kappa_n|}\right)^{1+p_n}<+\infty
\]
for every $R>0$, then the infinite product 
\[
\mathrm{Wei}_{\vec{\kappa},\vec{p}}(z)=\prod_{n=1}^\infty \left(1-\frac{z}{\kappa_n}\right)e^{P_{p_n}(z/\kappa_n)}
\]
converges and defines an entire function which has a zero at each point $\kappa_n$ and no other zeroes. More precisely, if $\kappa$ occurs with multiplicity $m$ in the sequence $\{\kappa_n\}$ then $\mathrm{Wei}_{\vec{\kappa},\vec{p}}(z)$ has a zero of order $m$ at $z=\kappa$. Moreover, the infinite product defining $\mathrm{Wei}_{\vec{\kappa},\vec{p}}(z)$ is unchanged under simultaneous renumbering of $\{\kappa_n\}$ and $\{p_n\}$. 
If $\Lambda\subseteq \mathbb{C}$ is a lattice, then the series
\begin{equation}\label{eq:converges}
\sum_{\lambda\in \Lambda\setminus\{0\}} \frac{1}{|\lambda|^s} 
\end{equation}
converges for $\Re(s)>2$. This implies that for any $r\geq 2$ the
series 
\[
\sum_{\lambda\in \Lambda\setminus\{0\}} \left( \frac{R}{|\lambda|}\right)^{1+r} = R^{r+1}\sum_{\lambda\in \Lambda\setminus\{0\}} \frac{1}{|\lambda|^{r+1}}
\]
converges, and so, by choosing $p_n$ to be the constant sequence $p_n\equiv r$ one sees that the infinite product 
\[
\prod_{\lambda\in \Lambda\setminus\{0\}}\left(1+\frac{z}{\lambda}\right)e^{P_r(-z/\lambda)}
\]
defines an entire function of $z$. Here we used the renumbering invariance to write the product as a product over $\Lambda\setminus\{0\}$. Now one makes a choice of arguments for the elements $\lambda\in \Lambda\setminus\{0\}$ in such a way that the $\zeta$-function $\zeta_{\Lambda\setminus\{0\}}$, defined by analytic extension of the holomorphic function
\[
\zeta_{\Lambda\setminus\{0\}}(s)=\sum_{\lambda\in \Lambda\setminus\{0\}}\frac{1}{\lambda^s}; \qquad \Re(s)>2
\]
is defined at $s=1$ and at $s=2$. By convergence of (\ref{eq:converges}) for $\Re(s)>2$, the $\zeta$-function $\zeta_{\Lambda\setminus\{0\}}$ will then be defined at every positive integer, and one can formally write
\begin{align*}
\prod_{\lambda\in \Lambda\setminus\{0\}}&\left(1+\frac{z}{\lambda}\right)=
\prod_{\lambda\in \Lambda\setminus\{0\}}e^{-P_r(-z/\lambda)}\left(1+\frac{z}{\lambda}\right)e^{P_r(-z/\lambda)}\\
&=
\left(\prod_{\lambda\in \Lambda\setminus\{0\}}e^{-P_r(-z/\lambda)}\right)
\prod_{\lambda\in \Lambda\setminus\{0\}}
\left(1+\frac{z}{\lambda}\right)e^{P_r(-z/\lambda)}\\
&=
e^{-\left(-\zeta_{\Lambda\setminus\{0\}}(1)z+\zeta_{\Lambda\setminus\{0\}}(2)\frac{z^2}{2}+\cdots+ \zeta_{\Lambda\setminus\{0\}}(r)\frac{(-z)^r}{r} \right)}
\prod_{\lambda\in \Lambda\setminus\{0\}}
\left(1+\frac{z}{\lambda}\right)e^{P_r(-z/\lambda)},
\end{align*}
where in the last step one has replaced the possibly divergent sums $\sum_{\lambda\in \Lambda\setminus\{0\}} \frac{1}{\lambda^j}$, for $j=1,\dots, r$ with their $\zeta$-regularizations. Thanks to absolute convergence, the terms $e^{(-z/\lambda)^k}$ freely move in and out from the product over $\Lambda\setminus\{0\}$ for $k>2$. Therefore,  
the last term in the above chain of formal identities is independent of $r$ as soon as $r\geq 2$ and we arrive at the following.
\begin{definition}\label{lemma-definition:weierstrass1}
Let a choice of arguments for the elements $\lambda\in \Lambda\setminus\{0\}$ in such a way that  $\zeta_{\Lambda\setminus\{0\}}$ is defined at $s=1$ and at $s=2$ be fixed. The Weierstra{\ss} $\zeta$-regularized product of the factors $(1+z/\lambda)$ with $\lambda$ ranging in $\Lambda\setminus\{0\}$ is
\begin{align*}
{\prod_{\lambda\in \Lambda\setminus\{0\}}}^{\hskip-6 pt \zeta}\left(1+\frac{z}{\lambda}\right)&:=
e^{\zeta_{\Lambda\setminus\{0\}}(1)z-\zeta_{\Lambda\setminus\{0\}}(2)\frac{z^2}{2}}
\prod_{\lambda\in \Lambda\setminus\{0\}}
\left(1+\frac{z}{\lambda}\right)e^{-\frac{z}{\lambda}+\frac{z^2}{2\lambda^2}}.
\end{align*}
\end{definition}
\begin{remark}
Choices of arguments for the elements $\lambda\in \Lambda\setminus\{0\}$ such that $\zeta_{\Lambda\setminus\{0\}}(1)$ and $\zeta_{\Lambda\setminus\{0\}}(2)$ are defined do actually exist and moreover there are quite natural choices with this property, see {\color{purple}\cite[Example 13]{quine}}. 
\end{remark}
We can now turn (\ref{eq:chern-weierstrass}) into a formal definition.
\begin{definition}\label{def:top-chern-wei}
Let $\nu_\Lambda$ be the normal bundle for the inclusion $X\hookrightarrow \mathrm{Maps}(\C/\Lambda,X)$. The Weierstra{\ss} $\zeta$-regularized equivariant top Chern class of $\nu_\Lambda\otimes\C$ in the antiholomorphic sector is defined as
\[
\widehat{c_{\mathrm{top},\C/\Lambda}^{\overline{\partial};\zeta}}(\nu_\Lambda\otimes \C):=\prod_{i=1}^{d}
{\prod_{\lambda\in \Lambda\setminus\{0\}}}^{\hskip-6 pt \zeta}\left(1+\frac{z}{\lambda}\right)\biggr\vert_{z=\alpha_{i}(X)\overline{\xi}_\Lambda^{-1}}
\]
\end{definition}
By analogy with Definition \ref{def:witten-classes} we then give the following.
\begin{definition}\label{def:regularized-witten-classes}
In the same assumptions as in Definitions \ref{lemma-definition:weierstrass1} {\color{purple}and \ref{def:top-chern-wei}},
 the \emph{$\zeta$-regularized total complex {\color{purple}$\mathcal{W}$-}class} of ${\color{purple}\nu\otimes\C}$ is the function
\[
{\color{purple}\mathcal{W}}^\zeta_{{\color{purple}\nu\otimes\C}}\colon\mathrm{Lattices}^+(\C)\to H^\bullet(X,\C)[\overline{\xi}_X^{-1}]
\]
given by
\[
{\color{purple}\mathcal{W}}^\zeta_{{\color{purple}\nu\otimes\C}}(\Lambda):=(j_\Lambda^\ast\otimes j_{\mathfrak{t};\Lambda}^*)\left(\frac{1}{\widehat{c_{\mathrm{top},\C/\Lambda}^{\overline{\partial};\zeta}}(\nu_\Lambda\otimes \C)}\right).
\]
The coefficient
\[
{\color{purple}\mathcal{W}}^\zeta_{{\color{purple}\nu\otimes\C};k}\colon \mathrm{Lattices}^+(\C)\to H^{2k}(X,\C)
\]
defined by the expansion 
\[
{\color{purple}\mathcal{W}}^\zeta_{{\color{purple}\nu\otimes\C}}(\Lambda)=\sum_{k=0}^\infty {\color{purple}\mathcal{W}}^\zeta_{{\color{purple}\nu\otimes\C};k}(\Lambda)\overline{\xi}_X^{-k}
\]
will be called the $k$-th $\zeta$-regularized complex {\color{purple}$\mathcal{W}$-}class of ${\color{purple}\nu\otimes\C}$.
\end{definition}
\begin{lemma}\label{lemma-definition:weierstrass2}
In the same assumptions as in Definition \ref{lemma-definition:weierstrass1} one has
\begin{align*}
{\prod_{\lambda\in \Lambda\setminus\{0\}}}^{\hskip-6 pt \zeta}\left(1+\frac{z}{\lambda}\right)=
e^{\zeta_{\Lambda\setminus\{0\}}(1)z-\zeta_{\Lambda\setminus\{0\}}(2)\frac{z^2}{2}}\frac{\sigma_\Lambda(z)}{z},
\end{align*}
where $\sigma_\Lambda(z)$ is the Weierstra{\ss} $\sigma$-function of the lattice $\Lambda$.
\end{lemma}
\begin{proof}
By definition of the Weierstra{\ss} $\sigma$-function (see, e.g., \cite[Section 20.42]{acoma} {\color{purple}or \cite[IV.3]{AEC})\footnote{{\color{purple} Beware that in the literature one has different conventions for the Weierstra{\ss} $\sigma$-function; compare, for instance, \cite[Chapter I, Theorem 6.4]{ataec} with \cite[Proposition 10.9]{ahr})}}}, one has
\begin{equation}\label{eq:sigma}
\sigma_\Lambda(z)=z \prod_{\lambda\in \Lambda\setminus\{0\}}
\left(1-\frac{z}{\lambda}\right)e^{\frac{z}{\lambda}+\frac{z^2}{2\lambda^2}}. 
\end{equation}
The statement then follows by changing $\lambda$ in $-\lambda$ in the above product and by comparing with Definition \ref{lemma-definition:weierstrass1}.
\end{proof}
\begin{corollary}
In the same assumptions as in Definition \ref{lemma-definition:weierstrass1} one has
\begin{equation}\label{eq:towards-witten-genus}
{\color{purple}\mathcal{W}}^\zeta_{{\color{purple}\nu\otimes\C}}(\Lambda)=e^{\zeta_{\Lambda\setminus\{0\}}(2)p_1(TX)\overline{\xi}_X^{-2}}
\prod_{i=1}^{d}\frac{z}{\sigma_\Lambda(z)}\biggr\vert_{z=\alpha_{i}(X)\overline{\xi}_X^{-1}},
\end{equation}
where $p_1(TX)$ denotes the first Pontryagin class of $TX$ seen as an element in $H^4(X;\C)$.
\end{corollary}
\begin{proof}
From Definitions \ref{lemma-definition:weierstrass1} and  \ref{def:top-chern-wei} and from Lemma \ref{lemma-definition:weierstrass2}, recalling Remark \ref{rem:j-is-identity} and that $j^\ast_{\mathfrak{t};\Lambda}(\overline{\xi}_\Lambda)=\overline{\xi}_X$, one has
\begin{equation}\label{eq:to-be-used-immediately}
{\color{purple}\mathcal{W}}^\zeta_{{\color{purple}\nu\otimes\C}}(\Lambda)= \prod_{i=1}^{d}\left(
e^{-\zeta_{\Lambda\setminus\{0\}}(1)z+\zeta_{\Lambda\setminus\{0\}}(2)\frac{z^2}{2}}\frac{z}{\sigma_\Lambda(z)}\right)\biggr\vert_{z=\alpha_{i}(X)\overline{\xi}_X^{-1}}.  
\end{equation}
One then rewrites the right hand side of (\ref{eq:to-be-used-immediately}) as
\[
e^{-\zeta_{\Lambda\setminus\{0\}}(1)c_1(TX\otimes \C)\overline{\xi}_X^{-1}+\zeta_{\Lambda\setminus\{0\}}(2)\left(\frac{1}{2}c_1(TX\otimes \C)^2-c_2(TX\otimes\C)\right)\overline{\xi}_X^{-1}}\prod_{i=1}^{d}\frac{z}{\sigma_\Lambda(z)}\biggr\vert_{z=\alpha_{i}(X)\overline{\xi}_X^{-1}},
\]
recalls that the odd Chern classes of the complexification of a real vector bundle vanish, and uses the relation $c_2(TX\otimes_\R\C)=-p_1(TX)$ to conclude.
\end{proof}

It is important to stress that the $\zeta$-function $\zeta_{\Lambda\setminus\{0\}}$ and its value at $2$ depend on the choice of arguments for the elements $\lambda$ in $\Lambda\setminus\{0\}$. Therefore, also $\mathcal{W}^\zeta_{\nu\otimes\C}$ depends on this choice. One removes this dependence by requiring that $p_1(X)$ is zero in $H^4(X;\mathbb{Q})$, i.e., by requiring that $X$ is a rational string manifold. With this assumption, formula (\ref{eq:towards-witten-genus}) reduces to
\begin{equation}\label{eq:witten-as-sigma}
{\color{purple}\mathcal{W}}^\zeta_{{\color{purple}\nu\otimes\C}}(\Lambda)=\prod_{i=1}^{d}\frac{z}{\sigma_\Lambda(z)}\biggr\vert_{z=\alpha_{i}(X)\overline{\xi}_X^{-1}},
\end{equation}
where now the right hand side is a canonically defined equivariant cohomology class in the antiholomorphic sector.

The entire function $z/\sigma_\Lambda(z)$ is the characteristic power series for the complex Witten genus \cite{ahr}. Therefore, summing up, we have obtained the following.
\begin{proposition}\label{prop:witten}
Let $X$ be a $d$-dimensional rational string manifold.
Then ${\color{purple}\mathcal{W}}^\zeta_{{\color{purple}\nu\otimes\C}}$ is the {\color{purple}complex} Witten class of the manifold $X$. In particular, if $d$ is even then
\[
\int_X {\color{purple}\mathcal{W}}^\zeta_{{\color{purple}\nu\otimes\C};d/2}
\]
is the complex Witten genus of $X$.
\end{proposition}
\begin{remark}\label{rem:reg-witten-classes-are modular}
Proposition \ref{prop:witten} in particular tells us that if $X$ is a rational string manifold, then the 
$\zeta$-regularized complex {\color{purple}$\mathcal{W}$-}classes 
\[
{\color{purple}\mathcal{W}}^\zeta_{{\color{purple}\nu\otimes\C};k}\colon \mathrm{Lattices}^+(\C)\to H^{2k}(X,\C)
\]
are modular forms of weight $k$. This can be directly seen from \eqref{eq:witten-as-sigma} by the modular properties of the function $(\Lambda,z)\mapsto z/\sigma_\Lambda(x)$, which are in turn immediate from the product formula \eqref{eq:sigma}.
\end{remark}

\begin{remark}\label{rem:choice-of-arguments}
When $\Lambda=\Z\oplus \Z\tau$ with $\Im(\tau)>0$, with the standard choice of arguments $-\pi\leq \mathrm{arg}(\lambda)<\pi$, one gets 
\[
\zeta_{\Lambda\setminus\{0\}}(2)=-4\pi i \frac{\eta'(\tau)}{\eta(\tau)}=G_2(\tau)
\]
where $\eta$ is the Dedekind $\eta$-function and $G_2$ is the quasi-modular Eisenstein series
\[
G_2(\tau)=\frac{\pi^2}{3}+\sum_{n \in \Z\setminus{0}} \sum_{m \in \Z}\dfrac{1}{(m+n\tau)^2},
\]
see \cite[Chapter 3, Ex.1]{apostol2012modular} and \cite[Example 13]{quine}. This explains the exponential prefactor
\[
e^{-G_2(\tau)p_1(X)}
\]
appearing in the expression of the Witten class for a non rationally string manifold for lattices of the standard form $\Z\oplus \Z\tau$. More generally, once an oriented basis $(\omega_1,\omega_2)$ for the lattice $\Lambda$ is chosen, one can write $\Lambda=\omega_1^{-1}(\Z\oplus\tau\Z)$ with $\tau=\omega_2/\omega_1$, choose an argument for $\omega_1$ in $[-\pi,\pi)$ and choose the arguments of the elements $\lambda\in \Lambda\setminus\{0\}$ so that 
\[
-\pi+\mathrm{arg}(\omega_1)\leq \mathrm{arg}(\lambda)<\pi+\mathrm{arg}(\omega_1).
\]
With this choice one has $\zeta_{\Lambda\setminus\{0\}}(2)=\omega_1^{-2}G_2(\tau)$.
\end{remark}

As an immediate corollary, we get the analogous of Proposition \ref{prop:witten} for the real {\color{purple}$\mathcal{W}$-}class. We first need a couple of obvious definitions.
\begin{definition}\label{def:euler-regularized}
Let $\nu_\Lambda$ be the normal bundle for the inclusion $X\hookrightarrow \mathrm{Maps}(\C/\Lambda,X)$. The Weierstra{\ss} $\zeta$-regularized equivariant Euler class of $\nu_\Lambda$ in the antiholomorphic sector is defined as
\[
\mathrm{eul}_{\C/\Lambda}^{\overline{\partial};\zeta}(\nu_\Lambda):=\sqrt{\widehat{c_{\mathrm{top},\C/\Lambda}^{\overline{\partial};\zeta}}(\nu_\Lambda\otimes \C)},
\]
where the determination of the square root is such that $\sqrt{1+\cdots}=1+\cdots$.
\end{definition}

\begin{definition}\label{def:regularized-witten-classes-real}
In the same assumptions as in Definitions \ref{lemma-definition:weierstrass1} {\color{purple}and \ref{def:top-chern-wei}},
 the \emph{$\zeta$-regularized total real {\color{purple}$\mathcal{W}$-}class} of $\nu$ is the function
\[
{\color{purple}\mathcal{W}}^\zeta_{\mathbb{R};{\color{purple}\nu}}\colon\mathrm{Lattices}^+(\C)\to H^\bullet(X,\C)[\overline{\xi}_X^{-1}]
\]
given by
\[
{\color{purple}\mathcal{W}}^\zeta_{\mathbb{R};{\color{purple}\nu}}(\Lambda):=(j_\Lambda^\ast\otimes j_{\mathfrak{t};\Lambda}^*)\left(\frac{1}{\widehat{\mathrm{eul}^{\overline{\partial};\zeta}_{\C/\Lambda}}(\nu_\Lambda)}\right).
\]
The coefficient
\[
{\color{purple}\mathcal{W}}^\zeta_{\mathbb{R};{\color{purple}\nu};k}\colon \mathrm{Lattices}^+(\C)\to H^{4k}(X,\C)
\]
defined by the expansion 
\[
{\color{purple}\mathcal{W}}^\zeta_{\mathbb{R};{\color{purple}\nu}}(\Lambda)=\sum_{k=0}^\infty {\color{purple}\mathcal{W}}^\zeta_{\mathbb{R};{\color{purple}\nu};k}(\Lambda)\xi_X^{-2k}
\]
will be called the $k$-th $\zeta$-regularized real {\color{purple}$\mathcal{W}$-}class of ${\nu}$.
\end{definition}
From equation \eqref{eq:witten-as-sigma} and by the usual $\sqrt{z}\leftrightarrow z$ rule for passing from Pontryagin classes of a real vector bundle to Chern classes of its complexification in characteristic power series for genera (see, e.g., \cite[Section 1.3]{Hirzebruch-topological-methods-in-algebraic-geometry}), we have that if $X$ is a rational string manifold of even dimension $d$, then 
\begin{equation}\label{eq:witten-as-sigma-real}
{\color{purple}\mathcal{W}}^\zeta_{\mathbb{R};{\color{purple}\nu}}(\Lambda)=\prod_{i=1}^{d/2}\frac{\sqrt{z}}{\sigma_\Lambda(z)}\biggr\vert_{z=\beta_{i}(X)\overline{\xi}_X^{-2}},
\end{equation}
where the $\beta_i(X)$ are the Pontryagin roots of $TX$. From this, one has the real {\color{purple}$\mathcal{W}$-}classes analogue of Proposition \ref{prop:witten}.

\begin{proposition}\label{prop:witten-real}
Let $X$ be a rational string manifold of even dimension $d$.
Then ${\color{purple}\mathcal{W}}^\zeta_{\mathbb{R};{\color{purple}\nu}}$ is the {\color{purple}real} Witten class of the manifold $X$. In particular, if $d\equiv 0\mod 4$ then
\[
\int_X {\color{purple}\mathcal{W}}^\zeta_{\mathbb{R};{\color{purple}\nu};d/2}
\]
is the real Witten genus of $X$.
\end{proposition}
Clearly, we also have the analogue of Remark \ref{rem:reg-witten-classes-are modular}
\begin{remark}\label{rem:reg-real-witten-classes-are modular}
If $X$ is a rational string manifold of even dimension $d$, then the 
$\zeta$-regularized real {\color{purple}$\mathcal{W}$-}classes 
\[
{\color{purple}\mathcal{W}}^\zeta_{\mathbb{R};{\color{purple}\nu};k}\colon \mathrm{Lattices}^+(\C)\to H^{4k}(X,\C)
\]
are modular forms of weight $2k$.
\end{remark}

\nocite{*}
\bibliographystyle{alpha}
\bibliography{bibliography}
\end{document}